\RequirePackage{fix-cm}
\documentclass[oneside,english]{amsart}
\usepackage[latin9]{inputenc}
\usepackage{geometry}
\usepackage{babel}
\usepackage{amsthm}
\usepackage{amssymb}
\usepackage{fixltx2e}
\usepackage[unicode=true,
 bookmarks=true,bookmarksnumbered=false,bookmarksopen=false,
 breaklinks=false,pdfborder={0 0 1},backref=false,colorlinks=false]
 {hyperref}
\hypersetup{pdftitle={Bachar ALHAJJAR}}

\makeatletter
\numberwithin{equation}{section}
\numberwithin{figure}{section}
\@ifundefined{lettrine}{\usepackage{lettrine}}{}
\theoremstyle{plain}
\newtheorem{thm}{\protect\theoremname}[section]
  \theoremstyle{definition}
  \newtheorem{defn}[thm]{\protect\definitionname}
  \theoremstyle{plain}
  \newtheorem{lem}[thm]{\protect\lemmaname}
  \theoremstyle{plain}
  \newtheorem{prop}[thm]{\protect\propositionname}
  \theoremstyle{plain}
  \newtheorem{cor}[thm]{\protect\corollaryname}
  \theoremstyle{remark}
  \newtheorem{rem}[thm]{\protect\remarkname}

\theoremstyle{definition}

\makeatother

  \providecommand{\corollaryname}{Corollary}
  \providecommand{\definitionname}{Definition}
  \providecommand{\lemmaname}{Lemma}
  \providecommand{\propositionname}{Proposition}
  \providecommand{\remarkname}{Remark}
\providecommand{\theoremname}{Theorem}

\begin{document}

\title{{\LARGE $\mathbb{\mathrm{\mathrm{LND}}}$-Filtrations and Semi-Rigid
Domains}}

\author{{\Large Bachar~~ALHAJJAR} }
\begin{abstract}
We investigate the filtration corresponding to the degree function
induced by a non-zero locally nilpotent derivation and its associated
graded algebra. We show that this kind of filtration, referred to
as the {\footnotesize $\mathrm{LND}$}\emph{-filtration}, is the ideal
candidate to study the structure of \emph{semi-rigid} $k$-domains,
that is, $k$-domains for which every non-zero locally nilpotent derivation
gives rise to the same filtration\emph{.} Indeed, the{\footnotesize{}
$\mathrm{LND}$}-filtration gives a very precise understanding of
these structure, it is impeccable for the computation of the Makar-Limanov
invariant, and it is an efficient tool to determine their isomorphism
types and automorphism groups. Then, we construct a new interesting
class of semi-rigid $k$-domains in which we elaborate the fundamental
requirement of {\footnotesize $\mathrm{LND}$}-filtrations. The importance
of these new examples is due to the fact that they possess a relatively
big set of invariant sub-algebras, which can not be recoverd by known
invariants such as the Makar-Limanov and the Derksen invariants. Also,
we define a new family of invariant sub-algebras as a generalization
of the Derksen invariant. Finally, we introduce an algorithm to establish
explicit isomorphisms between cylinders over non-isomorphic members
of the new class, providing by that new counter-examples to the cancellation
problem.
\end{abstract}

\address{Bachar ALHAJJAR, Institut de Math\'ematiques de Bourgogne, Universit\'e
de Bourgogne, 9 Avenue Alain Savary, BP 47870, 21078 Dijon, France.}

\email{Bachar.Alhajjar@gmail.com~,~~and~~~Bachar.Alhajjar@u-bourgogne.fr}



\maketitle

\section*{\textbf{\normalsize Introduction}}

Let $k$ be a field of characteristic zero and let $B$ be a commutative
$k$-domain. A $k$-derivation $\partial\in\mathrm{Der}_{k}(B)$ is
said to be\emph{ locally nilpotent} if for every $a\in B$, there
is an integer $n\geq0$ such that $\partial^{n}(a)=0$.

An important invariant of $k$-domains $B$ admitting non-trivial
locally nilpotent derivations is the so called \emph{Makar-Limanov
invariant} $\mathrm{ML}(B)$ which was defined by Makar-Limanov as
the intersection of the kernels of all locally nilpotent derivations
on $B$ (\cite{key-M-L1996}). This invariant was initially introduced
as a tool to distinguish certain $k$-domains from polynomial rings
but it has many other applications for the study of $k$-algebras
and their automorphism groups (\cite{key-ML2001}). One of the main
difficulties in applications is to compute this invariant without
a prior knowledge of all locally nilpotent derivations of a given
$k$-domain. 

In \cite{key-Kaliman Makar}, S. Kaliman and L. Makar-Limanov developed
general techniques to determine the $\mathrm{ML}$-invariant for a
class of finitely generated $k$-domains $B=k[X_{1},\ldots,X_{n}]/I$.
The idea, referred to as \textquotedblleft{}homogenization of derivations'',
is to reduce the problem to the study of homogeneous locally nilpotent
derivations on graded algebras $\mathrm{Gr}(B)$ associated to $B$.
For this, one considers appropriate filtrations $\mathcal{F}=\{\mathcal{F}_{i}\}_{i\in\mathbb{R}}$
on $B$ generated by real-valued weight degree functions $\omega\in\mathbb{R}^{n}$,
in such a way that every non-zero locally nilpotent derivation on
$B$ induces a non-zero homogeneous locally nilpotent derivation on
the associated graded algebra $\mathrm{Gr_{\mathcal{F}}}(B)$. Unfortunately,
these techniques only work if the associated graded algebra $\mathrm{Gr}(B)$
is in fact a $k$-domain itself, which will only occur if the ideal
$\hat{I}$, generated by top homogenous components relative to $\omega$
of all elements in $I$, is prime.

Finding a new way to tackle similar complications became an inevitable
necessity. Therefore, we start inspecting real-valued weight degree
functions on $k^{[N]}$ with the following new perspective. The positive
integer $N$ is chosen to be bigger than $n$, the dimension of the
ambient space. The considered ring $B=k^{[n]}/I$ is identified in
a specific ``twisting'' way to $k^{[N]}/J\simeq B$. This different
point of view allows us to avoid these kind of difficulties. Furthermore,
it simplifies the study of homogenous locally nilpotent derivation
even in these cases where classical techniques work. 

We present a new class of examples for which the \textquotedblleft{}homogenization''
method can be effectively applied with the alternative perspective,
while all other approaches fail. The new class comes to be a very
interesting object due to the fact that it possesses a huge set of
invariant sub-algebras, which can not be recoverd by any known invariant
such as the Makar-Limanov and the Derksen invariants.

As a modest outcome, the alternative approach delivers a full description
of the filtration induced by any locally nilpotent derivation, with
a finitely generated kernel, and its associated graded algebra. In
particular, a non-zero locally nilpotent derivation $\partial$ gives
rise to a proper $\mathbb{N}$-filtration $\mathcal{F}=\{\mathcal{F}_{i}\}_{i\in\mathbb{N}}$
of $B$ by the sub-spaces $\mathcal{F}_{i}=\mathrm{\ker}\partial^{i+1}$,
$i\in\mathbb{N}$. We call it the \emph{$\partial$-filtration}, which
corresponds to the $\mathbb{N}$-degree function $\deg_{\partial}$
induced by $\partial$. It turns out that $\deg_{\partial}$ is nothing
but the degree function induced by an $\mathbb{N}$-weight degree
function $\omega\in\mathbb{N}^{[N]}$ defined on $k^{[N]}$ for suitable
choices of $\omega$ and $N$.

In turn, the $\partial$-filtration comes out to be the ideal candidate
to study the structure of \emph{semi-rigid} $k$-domains, that is,
$k$-domains for which every non-zero locally nilpotent derivation
gives rise to the same filtration that we call the \emph{unique} $\mathrm{LND}$\emph{-filtration.}
This unique $\mathrm{LND}$-filtration gives a very precise understanding
of the structure of semi-rigid $k$-domains, it is impeccable for
the computation of the $\mathrm{ML}$-invariant, and it is an efficient
tool to determine isomorphism types and automorphism groups. Nevertheless,
the computation of the $\mathrm{ML}$-invariant, isomorphism types,
and automorphism groups of similar classically known structures can
be simplified and reduced considering this new point of view.

Another important tool for the study of non-rigid $k$-domains is
the\emph{ Derksen invariant} $\mathcal{D}(B)$ which is defined to
be the sub-algebra of $B$ generated by $\ker\partial$ for all non-zero
locally nilpotent derivations. We generalize this invariant to obtain
a new family $\left\{ \mathrm{\mathrm{AL}}_{i}(B)\right\} _{i\in\mathbb{N}}$
of invariant sub-algebras of $B$, where for each $i\in\mathbb{N}$
we define $\mathrm{\mathrm{AL}}_{i}(B)$ to be the algebra generated
by $\ker\partial^{i+1}$ for all non-zero locally nilpotent derivation
of $B$. We are interested in one particular member of this family
of invariants that corresponds to $i=1$, which we call the ring of
all local slices of $B$ and which we denote $\mathrm{AL}$-\emph{invariant}.
We show that the new class of $k$-domains can be realized as an affine
modification of the $\mathrm{AL}$-invariant with center $(f,I)$
for certain ideal $I$ in $\mathrm{AL}$ and some $f\in I$.

Finally, we propose an algorithm to construct explicit isomorphisms
between cylinders over non-isomorphic members of the new class, providing
by that new counter-examples to the cancellation problem.

\section{\textbf{Preliminaries}}

In this section we briefly recall basic facts on filtered algebra
and their relation with derivation in a form appropriate to our needs. 

In the sequel, unless otherwise specified $B$ will denote a commutative
domain over a field $k$ of characteristic zero. The set $\mathbb{Z}_{\geqslant0}$
of non-negative integers will be denoted by $\mathbb{N}$.

\subsection{\label{sub:Filtration-and-the} Filtrations and associated graded
algebras}
\begin{defn}
\label{def.proper.filtration} An \textit{$\mathbb{N}$-filtration}
of $B$ is a collection $\{\mathcal{F}_{i}\}_{i\in\mathbb{N}}$ of
$k$-sub-vector-spaces of $B$ with the following properties: 

1- $\mathcal{F}_{i}\subset\mathcal{F}_{i+1}$ for all $i\in\mathbb{N}$
. 

2- $B=\cup_{i\in\mathbb{N}}\mathcal{F}_{i}$ . 

3- $\mathcal{F}_{i}.\mathcal{F}_{j}\subset\mathcal{F}_{i+j}$ for
all $i,j\in\mathbb{N}$ . 

\noindent  The filtration is called \emph{proper} if the following
additional property holds: 

4- If $a\in\mathcal{F}_{i}\setminus\mathcal{F}_{i-1}$ and $b\in\mathcal{F}_{j}\setminus\mathcal{F}_{j-1}$,
then $ab\in\mathcal{F}_{i+j}\setminus\mathcal{F}_{i+j-1}$.
\end{defn}
\noindent  There is a one-to-one correspondence between proper $\mathbb{N}$-filtrations
and so called $\mathbb{N}$-degree functions:
\begin{defn}
An $\mathbb{N}$\emph{-degree function} on $B$ is a map $\deg:B\longrightarrow\mathbb{N}\mathbb{\cup}\{-\infty\}$
such that, for all $a,b\in B$, the following conditions are satisfied: 

$($1$)$ $\deg(a)=-\infty$ $\Leftrightarrow$ $a=0$.

$($2$)$ $\deg(ab)=\deg(a)+\deg(b)$.

$($3$)$ $\deg(a+b)\leq\max\{\deg(a),\deg(b)\}$.

\noindent  If the equality in (2) replaced by the inequality $\deg(ab)\leq\deg(a)+\deg(b)$,
we say that $\deg$ is an $\mathbb{N}$\emph{-semi-degree function.}
\end{defn}
Indeed, for an $\mathbb{N}$-degree function on $B$, the sub-sets
$\mathcal{F}_{i}=\{b\in B\mid\deg(b)\leq i\}$ are $k$-subvector
spaces of $B$ that give rise to a proper $\mathbb{N}$-filtration
$\{\mathcal{F}_{i}\}_{i\in\mathbb{N}}$. Conversely, every proper
$\mathbb{N}$-filtration $\{\mathcal{F}_{i}\}_{i\in\mathbb{N}}$,
yields an $\mathbb{N}$-degree function $\omega:B\longrightarrow\mathbb{N}\mathbb{\cup}\{-\infty\}$
defined by $\omega(0)=-\infty$ and $\omega(b)=i$ if $b\in\mathcal{F}_{i}\setminus\mathcal{F}_{i-1}$.
\begin{defn}
\label{Def.Graded.lg} Given a $k$-domain $B=\cup_{i\in\mathbb{N}}\mathcal{F}_{i}$
equipped with a proper $\mathbb{N}$-filtration, the associated graded
algebra $\mathrm{Gr}(B)$ is the $k$-vector space 
\[
\mathrm{Gr}(B)=\oplus_{i\in\mathbb{N}}\mathcal{F}_{i}/\mathcal{F}_{i-1}
\]
equipped with the unique multiplicative structure for which the product
of the elements $a+\mathcal{F}_{i-1}\in\mathcal{F}_{i}/\mathcal{F}_{i-1}$
and $b+\mathcal{F}_{j-1}\in\mathcal{F}_{j}/\mathcal{F}_{j-1}$, where
$a\in\mathcal{F}_{i}$ and $b\in\mathcal{F}_{j}$, is the element
\[
(a+\mathcal{F}_{i-1})(b+\mathcal{F}_{j-1}):=ab+\mathcal{F}_{i+j-1}\in\mathcal{F}_{i+j}/\mathcal{F}_{i+j-1}.
\]
Property 4 for a proper filtration in Definition \ref{def.proper.filtration}
ensures that $\mathrm{Gr}(B)$ is a commutative $k$-domain when $B$
is an integral domain. \noindent  Since for each $a\in B\setminus\{0\}$
the set $\{n\in\mathbb{N}\mid a\in\mathcal{F}_{n}\}$ has a minimum,
there exists $i$ such that $a\in\mathcal{F}_{i}$ and $a\notin\mathcal{F}_{i-1}$.
So we can define a $k$-linear map $\mathrm{gr}:B\longrightarrow\mathrm{Gr}(B)$
by sending $a$ to its class in $\mathcal{F}_{i}/\mathcal{F}_{i-1}$,
i.e $a\mapsto a+\mathcal{F}_{i-1}$, and $\mathrm{gr}(0)=0$. We will
frequently denote $\mathrm{gr}(a)$ simply by $\overline{a}$. Observe
that $\mathrm{gr}(a)=0$ if and only if $a=0$. 
\end{defn}
Denote by $\deg$ the $\mathbb{N}$-degree function $\deg:B\longrightarrow\mathbb{N}\cup\{-\infty\}$
corresponding to the proper $\mathbb{N}$-filtration $\{\mathcal{F}_{i}\}_{i\in\mathbb{N}}$.
We have the following properties.
\begin{lem}
\label{lem:graded relation} Given $a,b\in B$ the following holds:

\emph{P1)} $\overline{a\, b}=\overline{a}\,\overline{b}$, that is
$\mathrm{gr}$ is a multiplicative map. 

\emph{P2) }If $\deg(a)>\deg(b)$, then $\overline{a+b}=\overline{a}$. 

\emph{P3)} If $\deg(a)=\deg(b)=\deg(a+b)$, then $\overline{a+b}=\overline{a}+\overline{b}$. 

\emph{P4)} If $\deg(a)=\deg(b)>\deg(a+b)$, then $\overline{a}+\overline{b}=0$.
In particular, $\mathrm{gr}$ is not an additive map in general.\end{lem}
\begin{proof}
Let us assume that $\deg(a)=i$ and $\deg(b)=j$. By definition, $\deg(ab)=i+j$
means that $ab\in\mathcal{F}_{i+j}$ and $ab\notin\mathcal{F}_{i+j-1}$,
so $\overline{ab}=ab+\mathcal{F}_{i+j-1}:=(a+\mathcal{F}_{i-1})(b+\mathcal{F}_{j-1})=\overline{a}\,\overline{b}$.
Which gives P1. For P2 we observe that since $\deg(a+b)=\deg(a)$,
we have $\overline{a+b}=(a+b)+\mathcal{F}_{i-1}=(a+\mathcal{F}_{i-1})+(b+\mathcal{F}_{i-1})$,
and since $\mathcal{F}_{j-1}\subset\mathcal{F}_{j}\subseteq\mathcal{F}_{i-1}$
as $i>j$, we get $b+\mathcal{F}_{i-1}=0$. P3) is immediate, by definition.
Finally, assume by contradiction that $\overline{a}+\overline{b}\neq0$,
then $\overline{a}+\overline{b}=(a+\mathcal{F}_{i-1})+(b+\mathcal{F}_{i-1})=((a+b)+\mathcal{F}_{i-1})\neq0$,
which means that $a+b\notin\mathcal{F}_{i-1}$ and $\deg(a+b)=i$,
which is absurd. So P4 follows.
\end{proof}

\subsection{Derivations}

\indent\newline\noindent  By a \textit{$k$-derivation} of $B$,
we mean a $k$-linear map $D:B\longrightarrow B$ which satisfies
the Leibniz rule: For all $a,b\in B$; $D(ab)=aD(b)+bD(a)$. The set
of all $k$-derivations of $B$ is denoted by $\mathrm{Der}_{k}(B)$. 

\noindent  The \textit{kernel} of a derivation $D$ is the subalgebra
$\ker D=\left\{ b\in B;D(b)=0\right\} $ of $B$.

\noindent  The \textit{plinth ideal} of $D$ is the ideal $\mathrm{pl}(D)=\ker D\cap D(B)$
of $\mathrm{\ker}D$, where $D(B)$ denotes the image of $B$.

\noindent  An element $s\in B$ such that $D(s)\in\ker(D)\setminus\{0\}$
is called a \textit{local slice} for $D$.
\begin{defn}
\label{Def Graded} Given a $k$-algebra $B=\cup_{i\in\mathbb{N}}\mathcal{F}_{i}$
equipped with a proper $\mathbb{N}$-filtration, a $k$-derivation
$D$ of $B$ is said to \emph{respect} the filtration if there exists
an integer $d$ such that $D(\mathcal{F}_{i})\subset\mathcal{F}_{i+d}$
for all $i\in\mathbb{N}$. The smallest integer $d$, such that $D(\mathcal{F}_{i})\subset\mathcal{F}_{i+d}$
for all $i\in\mathbb{N}$, is called the \emph{degree }of\emph{ }$D$
with respect to $\mathcal{F}=\{\mathcal{F}_{i}\}_{i\in\mathbb{N}}$
and denoted by $\deg_{\mathcal{F}}D$.

\noindent  Note that if $D$ respects the filtration $\mathcal{F}=\{\mathcal{F}_{i}\}_{i\in\mathbb{N}}$
then $\deg_{\mathcal{F}}D$ is well-defined. Indeed, denote by $\deg$
the $\mathbb{N}$-degree function corresponding to $\mathcal{F}=\{\mathcal{F}_{i}\}_{i\in\mathbb{N}}$
and let $U$ be the non-empty subset of $\mathbb{Z}\cup\{-\infty\}$
defined by $U:=\left\{ \deg\left(D(b)\right)-\deg\left(b\right)\,;\, b\in B\setminus\{0\}\right\} $.
Since $D$ respects the filtration $\mathcal{F}=\{\mathcal{F}_{i}\}_{i\in\mathbb{N}}$,
the set $U$ is bounded above by $d$. Thus it has a greatest element
$d_{0}$ which is exactly $\deg_{\mathcal{F}}D$ by definition. 

\noindent  Suppose that $D$ respects the filtration $\mathcal{F}=\{\mathcal{F}_{i}\}_{i\in\mathbb{N}}$
and let $d=\deg_{\mathcal{F}}D$, we define a $k$-linear map $\overline{D}:\mathrm{Gr}(B)\longrightarrow\mathrm{Gr}(B)$
as follows: If $D=0$, then $\overline{D}=0$ the zero map. Otherwise,
if $D\neq0$ then we define 
\[
\overline{D}:\mathcal{F}_{i}/\mathcal{F}_{i-1}\longrightarrow\mathcal{F}_{i+d}/\mathcal{F}_{i+d-1}
\]
by the rule $\overline{D}(a+\mathcal{F}_{i-1})=D(a)+\mathcal{F}_{i+d-1}$.
Now extend $\overline{D}$ to all of $\mathrm{Gr}(B)$ by linearity.
One checks that $\overline{D}$ satisfies the Leibniz rule, therefore
it is a homogeneous $k$-derivation of $\mathrm{Gr}(B)$ of degree
$d$, that is, $\overline{D}$ sends homogeneous elements of degree
$i$ to zero or to homogeneous elements of degree $i+d$ . 
\end{defn}
Observe that $\overline{D}=0$ if and only if $D=0$. In addition,
$\mathrm{gr}(\ker D)\subset\ker\overline{D}$.

\section{$\mathrm{LND}$\textbf{-Filtrations and New Invariant sub-algebras}}

In this section we introduce the $\partial$-filtration associated
with a locally nilpotent derivation $\partial$. We explain how to
compute this filtration and its associated graded algebra in certain
situations. Also we present new invariants that generalize the Derksen
invariant. 
\begin{defn}
A $k$-derivation $\partial\in\mathrm{Der}_{k}(B)$ is said to be\emph{
locally nilpotent} if for every $a\in B$, there exists $n\in\mathbb{N}$
$($depending of $a$$)$ such that $\partial^{n}(a)=0$. The set
of all locally nilpotent derivations of $B$ is denoted by $\mathrm{\mathrm{LND}}(B)$.

\noindent  In particular, every locally nilpotent derivation $\partial$
of $B$ gives rise to a proper $\mathbb{N}$-filtration of $B$ by
the sub-spaces $\mathcal{F}_{i}=\mathrm{\ker}\partial^{i+1}$, $i\in\mathbb{N}$,
that we call the \emph{$\partial$-filtration. }It is straightforward
to check (see \cite[ Prop. 1.9]{key-gene}) that the $\partial$-filtration
corresponds to the $\mathbb{N}$-degree function $\deg_{\partial}:B\longrightarrow\mathbb{N}\cup\{-\infty\}$
defined by 
\end{defn}
\begin{center}
$\deg_{\partial}(a):=\min\{i\in\mathbb{N}\mid\partial^{i+1}(a)=0\}$,
and $\deg_{\partial}(0):=-\infty$. 
\par\end{center}

\noindent  Note that by definition $\mathcal{F}_{0}=\ker\partial$
and that $\mathcal{F}_{1}\setminus\mathcal{F}_{0}$ consists of all
local slices for $\partial$.\\

Let $\mathrm{Gr}_{\partial}(B)=\oplus_{i\in\mathbb{N}}\mathcal{F}_{i}/\mathcal{F}_{i-1}$
denote the associated graded algebra relative to the $\partial$-filtration
$\{\mathcal{F}_{i}\}_{i\in\mathbb{N}}$. Let $\mathrm{gr}_{\partial}:B\longrightarrow\mathrm{Gr_{\partial}}(B)$;
$a\overset{\mathrm{gr}_{\partial}}{\longmapsto}\overline{a}$ be the
natural map between $B$ and $\mathrm{Gr_{\partial}}(B)$ defined
in \ref{Def.Graded.lg}, where $\overline{a}$ denote $\mathrm{gr}_{\partial}(a)$.

The next Proposition, which is due to Daigle (\cite[Theorem 2.11]{key-gene},
see also \cite[Theorem 1.7 and Corollary 4.12]{D. daigle: tame degree}),
implies in particular that if $B$ is of finite transcendence degree
over $k$, then every non-zero $D\in\mathrm{LND}(B)$ respects the
$\partial$-filtration and therefore induces a non-zero homogeneous
locally nilpotent derivation $\overline{D}$ of $\mathrm{Gr}_{\partial}(B)$.
\begin{prop}
\label{Pro:Daigle} $($\textbf{\emph{Daigle}}$)$ Suppose that $B$
is a commutative domain, of finite transcendence degree over $k$.
Then for every pair $D\in\mathrm{Der}_{k}(B)$ and $\partial\in\mathrm{\mathrm{LND}}(B)$,
$D$ respects the $\partial$-filtration. Consequently, $\overline{D}$
is a well defined homogeneous derivation of the integral domain $\mathrm{Gr}_{\partial}(B)$
relative to this filtration, and it is locally nilpotent if $D$ is
locally nilpotent.
\end{prop}

\subsection{New Invariants $($$\mathrm{\mathrm{AL}}_{i}$-invariants$)$}
\begin{defn}
\emph{\label{Def:The-AL-iinvariant}} Let $\partial\in\mathrm{\mathrm{LND}}(B)$
be non-zero and let $\{\mathcal{F}_{i}\}_{i\in\mathbb{N}}$ be the
$\partial$-filtration. We denote by $\mathrm{L}_{\partial}$ the
sub-algebra of $B$ generated by $\mathcal{F}_{1}=\ker\partial^{2}$
and we call it the\emph{ ring of local slices} for $\partial$. The
sub-algebra of $B$ generated by $\mathrm{L}_{\partial}$ for all
non-zero locally nilpotent derivation of $B$ will be called the ring
of \emph{all local slices} of $B$. It will be denoted by $\mathrm{\mathrm{AL}}(B)$
and referred to as the\emph{ $\mathrm{\mathrm{AL}}$-invariant}, which
is manifested by the fact that\emph{ }$\mathrm{\mathrm{AL}}(B)$ is
invariant by all algebraic $k$-automorphisms of $B$.
\end{defn}
\noindent  In a sense, the $\mathrm{\mathrm{AL}}$-invariant is a
generalization of the\emph{ Derksen invariant} $\mathcal{D}(B)$ which
is defined to be the sub-algebra of $B$ generated by $\ker\partial$
for all non-zero $\mathrm{\mathrm{LND}}(B)$. 

In a more general way we define the following invariants for non-rigid
$k$-domains. Let $\mathrm{\mathrm{AL}}_{i}(B)$ denotes the sub-algebra
of $B$ generated by $\ker\partial^{i+1}$ for all non-zero locally
nilpotent derivation of $B$, then it is clear that $\mathrm{\mathrm{AL}}_{i}(B)$
is invariant by all algebraic $k$-automorphisms of $B$. Indeed,
let $\partial\in\mathrm{LND}(B)\setminus\{0\}$ and $\alpha\in\mathrm{\mathrm{Aut}}_{k}(B)$,
then $(\alpha^{-1}\partial\alpha)^{n}=\alpha^{-1}\partial^{n}\alpha$
for every $n\in\mathbb{N}$, which implies that $\partial_{\alpha}:=\alpha^{-1}\partial\alpha\in\mathrm{LND}(B)$.
Therefore, $\alpha.\partial_{\alpha}^{n}=\partial^{n}.\alpha$ so
we get $\deg_{\partial_{\alpha}}(b)=\deg_{\partial}(\alpha(b))$ for
any $b\in B$. In other words, $\alpha$ respects $\deg_{\partial}$
and $\deg_{\partial_{\alpha}}$, that is, $\alpha$ sends an element
of degree $i$ relative to $\deg_{\partial_{\alpha}}$, to an element
of the same degree $i$ relative to $\deg_{\partial}$.

\noindent  Note that $\mathrm{\mathrm{AL}}_{0}(B)=\mathcal{D}(B)$,
$\mathrm{\mathrm{AL}}_{1}(B)=\mathrm{\mathrm{AL}}(B)$, and $\mathrm{\mathrm{AL}}_{i}(B)\subseteq\mathrm{\mathrm{AL}}_{i+1}(B)$
for all $i$. 

In the case where $B=k[X_{1},\ldots,X_{n}]/I=k[x_{1},\ldots,x_{n}]$
is a finitely generated $k$-domain that admits a non-zero $\partial\in\mathrm{LND}(B)$,
the chain of inclusions $\mathrm{\mathrm{AL}}_{0}(B)\hookrightarrow\mathrm{\mathrm{AL}}_{1}(B)\hookrightarrow\mathrm{\mathrm{AL}}_{2}(B)\hookrightarrow\cdots\hookrightarrow\mathrm{\mathrm{AL}}_{i}(B)\hookrightarrow\cdots$
will eventually stabilize, that is, there exists $d\in\mathbb{N}$
such that $\mathrm{\mathrm{AL}}_{d}(B)=B$. Indeed, by definition
of $\mathrm{LND}$, there exist $d_{1},\ldots,d_{n}\in\mathbb{N}$
such that $\partial^{d_{i}+1}(x_{i})=0$. Denote by $\{\mathcal{F}_{i}\}_{i\in\mathbb{N}}$
the $\partial$-filtration and let $d=\max_{i\in\{1,\ldots,n\}}\{d_{i}\}$,
then $x_{i}\in\mathcal{F}_{d}$ for all $i$. Since $\mathcal{F}_{d}\subset\mathrm{\mathrm{AL}}_{d}(B)$
by definition, we see that $B\subseteq\mathrm{\mathrm{AL}}_{d}(B)$
and we are done. 

Recall that the \emph{Makar-Limanov invariant} $\mathrm{ML}(B)$ is
defined to be the intersection of the kernels of all locally nilpotent
derivations on $B$. One might think that $\mathrm{\mathrm{AL}}_{i\in\mathbb{N}}$-invariants,
in addition to the $\mathrm{\mathrm{ML}}$-invariant, cover all invariant
sub-algebras of $B$. This, however, is not correct, see \S \ref{Sec:A-huge}
below.

\subsection{Computing the $\partial$-filtration and its associated graded algebra}

\indent\newline\noindent  Here, given a finitely generated $k$-domain
$B$, we describe a general method which enables the computation of
the $\partial$-filtration for a locally nilpotent derivation with
finitely generated kernel. First we consider a more general situation
where the plinth ideal $\mathrm{pl}(\partial)$ is finitely generated
as an ideal in $\ker\partial$ then we deal with the case where $\ker\partial$
is itself finitely generated as a $k$-algebra.

\subsubsection{\emph{Let $B=k[X_{1},\ldots,X_{n}]/I=k[x_{1},\ldots,x_{n}]$ be a
finitely generated $k$-domain, and let $\partial\in\mathrm{LND}(B)$
be such that $\mathrm{pl}(\partial)$ is generated by precisely $m$
elements say $f_{1},\ldots,f_{m}$ as an ideal in $\ker\partial$.
Denote by $\mathcal{F}=\{\mathcal{F}_{i}\}_{i\in\mathbb{N}}$ the
$\partial$-filtration, then:}}

By definition $\mathcal{F}_{0}=\ker\partial$. Furthermore, given
elements $s_{i}\in\mathcal{F}_{1}$ such that $\partial(s_{i})=f_{i}$
for every $i\in\{1,\ldots,m\}$, it is straightforward to check that 

\textit{\emph{
\begin{eqnarray*}
\mathcal{F}_{1} & = & \mathcal{F}_{0}s_{1}+\ldots+\mathcal{F}_{0}s_{m}+\mathcal{F}_{0}.
\end{eqnarray*}
}}Letting $\deg_{\partial}(x_{i})=d_{i}$, we denote by $H_{j}$ the
$\mathcal{F}_{0}$-sub-module in $B$ generated by elements of degree
$j$ relative to $\deg_{\partial}$ of the form $s_{1}^{u_{1}}\ldots s_{m}^{u_{m}}x_{1}^{v_{1}}\ldots x_{n}^{v_{n}}$,
that is, 
\[
H_{j}:=\sum_{\sum_{l}u_{l}+\sum_{i}d_{i}.v_{i}=j}\mathcal{F}_{0}\left(s_{1}^{u_{1}}\ldots s_{m}^{u_{m}}x_{1}^{v_{1}}\ldots x_{n}^{v_{n}}\right)
\]
 where $u_{l},v_{i}\in\mathbb{N}$ for all $i$ and $l$. The integer
$\sum_{l}u_{l}+\sum_{i}d_{i}v_{i}$ is nothing but $\deg_{\partial}(s_{1}^{u_{1}}s_{2}^{u_{2}}\ldots.s_{m}^{u_{m}}x_{1}^{v_{1}}x_{2}^{v_{2}}.\ldots.x_{n}^{v_{n}})$.
Then we define a new $\mathbb{N}$-filration $\mathcal{G}=\{\mathcal{G}_{i}\}_{i\in\mathbb{N}}$
of $B$ by setting 
\[
\mathcal{G}_{i}=\sum_{j\leq i}H_{j}.
\]
By construction $\mathcal{G}_{i}\subseteq\mathcal{F}_{i}$ for all
$i\in\mathbb{N}$, with equality for $i=0$ and $i=1$. The following
result provides a characterization of when these two filtrations coincide:
\begin{lem}
\label{lem:equal-if-it-is-proper} The filtrations $\mathcal{F}$
and $\mathcal{G}$ are equal if and only if $\mathcal{G}$ is proper.\end{lem}
\begin{proof}
One direction is clear since $\mathcal{F}$ is proper. Conversely,
suppose that $\mathcal{G}$ is proper with the corresponding $\mathbb{N}$-degree
function $\omega$ on $B$ (see \S \ref{sub:Filtration-and-the}).
Given $f\in\mathcal{F}_{i}\setminus\mathcal{F}_{i-1}$, $i>1$, for
every local slice $s\in\mathcal{F}_{1}\setminus\mathcal{F}_{0}$,
there\textit{\emph{ exist $f_{0}\neq0,a_{i}\neq0,a_{i-1},\ldots,a_{0}\in\mathcal{F}_{0}$
such that}} $f_{0}f=a_{i}s^{i}+a_{i-1}s^{i-1}+\cdots+a_{0}$ (see
\cite[Proof of Lemma 4]{key-Makar}). Since $\omega(g)=0$ (resp.
$\omega(g)=1$) for every $g\in\mathcal{F}_{0}$ (resp. $g\in\mathcal{F}_{1}\setminus\mathcal{F}_{0}$),
we obtain 
\[
\omega(f)=\omega(f_{0}f)=\omega(a_{i}s^{i}+a_{i-1}s^{i-1}+\cdots+a_{0})=\max\{\omega(a_{i}s^{i}),\ldots,\omega(a_{0})\}=i,
\]
 and so $f\in\mathcal{G}_{i}$.
\end{proof}

\subsubsection{\emph{The twisting technique}}

\indent\newline\noindent  Next, we determine the $\partial$-filtration,
for a locally nilpotent derivation $\partial$ with finitely generated
kernel, by giving an effective criterion to decide when the $\mathbb{N}$-filtration
$\mathcal{G}$ defined above is proper.

Hereafter, we assume that $0\in\mathrm{Spec}(B)$ and that $\ker(\partial)$
is generated by elements $z_{j}\in B$ such that $z_{j}(0,\ldots,0)=0$,
$j\in\{1,\ldots,r\}$. Since $\ker(\partial)$ is finitely generated
$k$-algebra, the plinth ideal \textit{\emph{$\mathrm{pl}(\partial)$
is finitely generated. So there exist $s_{1}$, $\ldots$, $s_{m}$
$\in\mathcal{F}_{1}$ such that }}$\mathcal{F}_{1}=\mathcal{F}_{0}s_{1}+\ldots+\mathcal{F}_{0}s_{m}+\mathcal{F}_{0}$.
We can also assume that $s_{i}$ is irreducible and $s_{i}(0,\ldots,0)=0$
for all $i$.

Letting $J\subset k^{[r+n+m]}=k[Z_{1},\ldots,Z_{r}][X_{1},\ldots,X_{n}][S_{1},\ldots,S_{m}]$
be the ideal generated by $I$ and the elements $Z_{j}=z_{j}(X_{1},\ldots,X_{n})$,
$j\in\{1,\ldots,r\}$, $S_{i}=s_{i}(X_{1},\ldots,X_{n})$, $i\in\{1,\ldots,m\}$,
then we have 
\[
B=k[Z_{1},\ldots,Z_{r}][X_{1},\ldots,X_{n}][S_{1},\ldots,S_{m}]/J.
\]
Note that by construction $0\in\mathrm{Spec}(k^{[r+n+m]}/J)$. 

We define an $\mathbb{N}$-weight degree function $\omega$ on $k^{[r+n+m]}$
by declaring that $\omega(Z_{i})=0=\deg_{\partial}(z_{i})$ for all
$i\in\{1,\ldots,r\}$, $\omega(S_{i})=1=\deg_{\partial}(s_{i})$ for
all $i\in\{1,\ldots,m\}$, and $\omega(X_{i})=\deg_{\partial}(x_{i})=d_{i}$
for all $i\in\{1,\ldots,n\}$. The corresponding proper $\mathbb{N}$-filtration
$\mathcal{Q}_{i}:=\{P\in k^{[n]}\mid\omega(P)\leq i\}$, $i\in\mathbb{N}$,
on $k^{[r+n+m]}$ has the form $\mathcal{\mathcal{Q}}_{i}=\oplus_{j\leq i}\mathcal{H}_{j}$
where 
\[
\mathcal{H}_{j}:=\oplus_{\sum_{j}u_{j}+\sum_{i}d_{i}v_{i}=j}k[Z_{1},\ldots,Z_{r}]S_{1}^{u_{1}}\ldots S_{m}^{u_{m}}X_{1}^{v_{1}}\ldots X_{n}^{v_{n}}.
\]
 By construction $\pi\left(\mathcal{\mathcal{Q}}_{i}\right)=\mathcal{G}_{i}$
where $\pi:k^{[r+n+m]}\longrightarrow B$ denotes the natural projection.
Indeed, since 
\[
\pi\left(\mathcal{\mathcal{Q}}_{i}\right)=\sum_{j\leq i}\pi\left(\mathcal{H}_{j}\right)
\]
 and $\pi\left(\mathcal{H}_{j}\right)=\sum_{\sum_{j}u_{j}+\sum_{i}d_{i}v_{i}=j}k[z_{1},\ldots,z_{r}]s_{1}^{u_{1}}\ldots.s_{m}^{u_{m}}x_{1}^{v_{1}}\ldots x_{n}^{v_{n}}$,
we get 
\[
\pi\left(\mathcal{H}_{j}\right)=\sum_{\sum_{j}u_{j}+\sum_{i}d_{i}v_{i}=j}\left(\ker\partial\right)s_{1}^{u_{1}}.\ldots s_{m}^{u_{m}}x_{1}^{v_{1}}.\ldots x_{n}^{v_{n}}
\]
 which means precisely that $\pi\left(\mathcal{\mathcal{Q}}_{i}\right)=\mathcal{G}_{i}$.

Let $\hat{J}\subset k^{[r+n+m]}$ be the homogeneous ideal generated
by the highest homogeneous components relative to $\omega$ of all
elements in $J$. Then we have the following result, which is inspired
by the technique developed by S. Kaliman and L. Makar-Limanov:
\begin{prop}
\label{Pro:when-it-is-proper} The $\mathbb{N}$-filration $\mathcal{G}$
is proper if and only if $\hat{J}$ is prime.\end{prop}
\begin{proof}
It is enough to show that $\mathcal{G}=\{\pi\left(\mathcal{\mathcal{Q}}_{i}\right)\}_{i\in\mathbb{N}}$
coincides with the filtration corresponding to the $\mathbb{N}$-semi-degree
function $\omega_{B}$ on $B$ defined by 
\[
\omega_{B}(p):=\min_{P\in\pi^{-1}(p)}\{\omega(P)\}.
\]
Indeed, if so, the result will follow from \cite[Lemma 3.2]{key-Kaliman Makar}
which asserts in particular that $\omega_{B}$ is an $\mathbb{N}$-degree
function on $B$ if and only if $\hat{J}$ is prime. Let $\{\mathcal{G}_{i}^{'}\}_{i\in\mathbb{N}}$
be the filtration corresponding to $\omega_{B}$. Given $f\in\mathcal{G}_{i}^{'}$
there exists $F\in\mathcal{\mathcal{Q}}_{i}$ such that $\pi(F)=f$,
which means that $\mathcal{G}_{i}^{'}\subset\pi\left(\mathcal{\mathcal{Q}}_{i}\right)$.
Conversely, it is clear that $\omega_{B}(z_{i})=\omega(Z_{i})=0$
for all $i\in\{1,\ldots,r\}$. Furthermore $\omega_{B}(s_{i})=\omega(S_{i})=1$
for all $i\in\{1,\ldots,m\}$, for otherwise $s_{i}\in\ker\partial$
which is absurd. Finally, if $d_{i}\neq0$ and $\omega_{B}(x_{i})<\omega(X_{i})=d_{i}\neq0$,
then $x_{i}\in\pi\left(\mathcal{\mathcal{Q}}_{d_{i}-1}\right)\subset\ker\partial^{d_{i}-1}$
which implies that $\deg_{\partial}(x_{i})<d_{i}$, a contradiction.
So $\omega_{B}(x_{i})=d_{i}$. Thus $\omega_{B}(f)\leq i$ for every
$f\in\pi\left(\mathcal{\mathcal{Q}}_{i}\right)$ which means that
$\pi\left(\mathcal{\mathcal{Q}}_{i}\right)\subset\mathcal{G}_{i}^{'}$.
\end{proof}
The next Proposition, which is a reinterpretation of \cite[Prop. 4.1]{key-Kaliman Makar},
describes the associated graded algebra $\mathrm{Gr}_{\partial}(B)$
of the filtered algebra $\left(B,\mathcal{F}\right)$ in the case
where the $\mathbb{N}$-filtration $\mathcal{G}$ is proper:
\begin{prop}
\label{prop:graded algebra} If the $\mathbb{N}$-filtration $\mathcal{G}$\textup{
}\textup{\emph{is proper then }}\textup{$\mathrm{Gr}_{\partial}(B)\simeq k^{[r+n+m]}/\hat{J}$.} \end{prop}
\begin{proof}
By virtue of $($\cite[Prop. 4.1]{key-Kaliman Makar}$)$ the graded
algebra associated to the filtered algebra $\left(B,\mathcal{G}\right)$
is isomorphic to $k^{[r+n+m]}/\hat{J}$. So the assertion follows
from Lemma \ref{lem:equal-if-it-is-proper}. 
\end{proof}

\section{\textbf{Semi-Rigid $k$-Domains}}

\subsection{Definitions and basic properties}

\indent\newline\noindent  In \cite{key-David Maubach}, D. Finston
and S. Maubach considered rings $B$ whose sets of locally nilpotent
derivations are \textquotedblleft{}one-dimensional\textquotedblright{}
in the sense that $\mathrm{LND}(B)=\ker(\partial).\partial$ for some
non-zero $\partial\in\mathrm{LND}(B)$. They called them \emph{almost-rigid}
rings. Hereafter, we consider the following definition which seems
more natural in our context (see Prop. \ref{Pro:semi-rigid one dimensional}
below for a comparison between the two notions). 
\begin{defn}
\label{Def:semi-rigid} A commutative domain $B$ over a field $k$
of characteristic zero is called \emph{semi-rigid} if all non-zero
locally nilpotent derivations of $B$ induce the same \textit{\emph{proper
$\mathbb{N}$-filtration $($equivalently, the same $\mathbb{N}$-degree
function$)$.}} 

\noindent  The unique proper \textit{\emph{$\mathbb{N}$-filtration
of a semi-rigid $k$-domain $B$, that corresponds to any non-zero
}}$\partial\in\mathrm{LND}(B)$, will be referred to and called the
\emph{unique }$\mathrm{LND}$\emph{-filtration} of $B$.
\end{defn}
Semi-rigid $k$-domains $B$ can be equivalently characterized in
terms of their \emph{Makar-Limanov invariant }$\mathrm{ML}(B):=\cap_{D\in\mathrm{\mathrm{LND}}(B)}\ker(D)$
as follows:
\begin{prop}
\label{prop:semi-rigid=00003DML=00003Dker} A $k$-domain $B$ is
semi-rigid if and only if $\mathrm{ML}(B)=\ker(\partial)$ for any
non-zero $\partial\in\mathrm{LND}(B)$.\end{prop}
\begin{proof}
Given $D,E\in\mathrm{LND}(B)\setminus\{0\}$ such that $A:=\ker(D)=\ker(E)$,
there exist non-zero elements $a,b\in A$ such that $aD=bE$ (\cite[Principle 12]{key-gene})
which implies that the $D$-filtration is equal to the $E$-filtration.
So if $\mathrm{ML}(B)=\ker(\partial)$ for any non-zero $\partial\in\mathrm{LND}(B)$
then $B$ is semi-rigid. The other implication is clear by definition. 
\end{proof}
Recall that $D\in\mathrm{Der}_{k}(B)$ is \textit{irreducible} if
and only if $D(B)$ is contained in no proper principal ideal of $B$,
and that $B$ is said to satisfy the ascending chain condition (ACC)
on principal ideals if and only if every infinite chain $(b_{1})\subset(b_{2})\subset(b_{3})\subset\cdots$
of principal ideals of $B$ eventually stabilizes. $B$ is said to
be a \textit{highest common factor ring}, or HCF-ring, if and only
if the intersection of any two principal ideals of $B$ is again principal.
\begin{prop}
\label{Pro:semi-rigid one dimensional} Let $B$ be a semi-rigid $k$-domain
satisfying the ACC on principal ideals. If $\mathrm{ML}(B)$ is an
HCF-ring, then there exists a unique irreducible $\partial\in\mathrm{LND}(B)$
up to multiplication by unit. Consequently, every $D\in\mathrm{LND}(B)$
has the form $D=f\partial$ for some $f\in\ker(\partial)$, and so
$B$ is almost rigid.\end{prop}
\begin{proof}
Existence follows from the fact that since $B$ satisfies the ACC
on principal ideals, then for every non-zero $T\in\mathrm{LND}(B)$,
there exists an irreducible $T_{0}\in\mathrm{LND}(B)$ and $c\in\ker(T)$
such that $T=cT_{0}$. (\cite[Prop. 2.2 and Principle 7]{key-gene}).
The argument for uniqueness is similar to that in \cite[ Prop. 2.2.b]{key-gene},
but with a little difference, that is, in \cite{key-gene} it is assumed
that $B$ itself is an HCF-ring while here we only require that $\mathrm{ML}(B)$
is an HCF-ring. Namely, let $D,E\in\mathrm{LND}(B)$ be irreducible
derivations, and let $A=\mathrm{ML}(B)$. By hypothesis $\ker(D)=\ker(E)=A$,
so there exist non-zero $a,b\in A$ such that $aD=bE$ (\cite[ Principle 12]{key-gene}).
Here we can assume that $a,b$ are not units otherwise we are done.
Set $T=aD=bE$. Since $A$ is an HCF-ring, there exists $c\in A$
such that $aA\cap bA=cA$. Therefore, $T(B)\subset cB$, and there
exists $T_{0}\in\mathrm{LND}(B)$ such that $T=cT_{0}$. Write $c=as=bt$
for $s,t\in B$. Then $cT_{0}=asT_{0}=aD$ implies $D=sT_{0}$, and
likewise $E=tT_{0}$. By irreducibility, $s$ and $t$ are units of
$B$, and we are done.
\end{proof}

\subsection{Elementary examples of semi-rigid $k$-domains}

\subsubsection{\textbf{\emph{Polynomial rings in one variable over rigid $k$-domains}}}

Recall that a $k$-domain $A$ is called \emph{rigid} if the zero
derivation is the only locally nilpotent $k$-derivation of $A$.
Equivalently, $A$ is rigid if and only if $\mathrm{ML}(A)=A$. The
next Proposition, which is due to Makar-Limanov (\cite[Lemma 21]{key-Makar},
also \cite[Theorem 3.1]{key-Crachiola Makar}), presents the simplest
examples of semi-rigid $k$-domains.
\begin{prop}
\textbf{\emph{$($Makar-Limanov$)$}} Let $A$ be a rigid domain of
finite transcendence degree over a field $k$ of characteristic zero.
Then the polynomial ring $A[x]$ is semi-rigid.\end{prop}
\begin{proof}
For the convenience of the reader, we provide an argument formulated
in the $\mathrm{LND}$-filtration language. Let $\partial$ be the
locally nilpotent derivation of $A[x]$ defined by $\partial(a)=0$
for every $a\in A$ and $\partial(x)=1$. Then the $\partial$-filtration
$\{\mathcal{F}_{i}\}_{i\in\mathbb{N}}$ is given by $\mathcal{F}_{i}=Ax^{i}\oplus\mathcal{F}_{i-1}$
where $\mathcal{F}_{0}=\ker(\partial)=A$, and the associated graded
algebra is $\mathrm{Gr}(A[x])=\oplus_{i\in\mathbb{N}}A\overline{x}^{i}$,
where $\overline{x}:=\mathrm{gr}(x)$ and $\overline{A}=A$. Since
$A[x]$ is of finite transcendence degree over $k$, Proposition \ref{Pro:Daigle}
implies that every non-zero $D\in\mathrm{LND}(A[x])$ respects the
$\partial$-filtration and induces a non-zero homogeneous locally
nilpotent derivation $\overline{D}$ of $\mathrm{Gr}(A[x])$ of a
certain degree $d=\deg_{\partial}(D)\geq-1$. It is enough to check
that in fact $d=-1$. Indeed, if so then $D=a\partial$, for some
$a\in A$ which implies the semi-rigidity of $A[x]$. So suppose for
contradiction that $d\geq0$, then $D(x)\in\mathcal{F}_{d+1}=Ax^{d+1}+\mathcal{F}_{d}$.
Therefore, $\overline{D}$ sends $\overline{x}$ to zero or to $\overline{a}\,\overline{x}^{d+1}$.
Either way, we have $\overline{x}\in\ker(\overline{D})$, see Corollary
1.20 \cite{key-gene}. Furthermore, $\overline{D}(\overline{a})=\begin{cases}
0\\
\overline{a}_{0}\overline{x}^{d}
\end{cases}$, so $\overline{D}=\overline{x}^{d}E$ where $E(\overline{a})=\begin{cases}
0\\
\overline{a}_{0}
\end{cases}$ and $E(\overline{x})=0$. This asserts that $E\in\mathrm{LND}(A[\overline{x}])$
by virtue of \cite[Principle 7]{key-gene}. Clearly, $E$ restricts
to $\mathrm{LND}(A)$, so by hypothesis $E=0$ which yields $\overline{D}=0$,
a contradiction.
\end{proof}

\subsubsection{\emph{\label{Sec:D.w.esample} }\textbf{\emph{Danielewski $k$-domains}}}

Let 
\[
B_{n,P}=k[X,S,Y]/\langle X^{n}Y-P(X,S)\rangle
\]
where $n\geq1$, $d\geq2$, $P(X,S)=S^{d}+f_{d-1}(X)S^{d-1}+\cdots+f_{0}(X),$
and $f_{i}(X)\in k[X]$. We call $B_{n,P}$ the Danielewski $k$-domain
corresponding to the pair $(n,P)$. Let $x$, $s$, $y$ be the class
of $X$, $S$, and $Y$ in $B_{n,P}$. It is well known (see \cite[Section 4]{key-Makar}
for the case $P\in k[S]$; and \cite[Section 2.4]{Ploni} for the
case, where $P(X,S)\in k[X,S]$) that if $n\geq2$, then $\mathrm{ML}(B_{n,P})=k[x]$
and $\mathrm{LND}(B_{n,P})=k[x].\partial$, for the locally nilpotent
derivation $\partial$ of $B$ defined by 
\[
\partial=x^{n}\partial_{s}+\frac{\partial P}{\partial s}\partial_{y},
\]
where $\frac{\partial P}{\partial s}=ds^{d-1}+(d-1)f_{d-1}(x)s^{d-2}+\ldots+f_{1}(x)$.
\textit{\emph{Hence, $B_{n,P}$ is almost rigid.}}

We easily recover these previous results using the $\mathrm{LND}$-filtration
method as follows:

\textit{\emph{The $\partial$-filtration }}$\{\mathcal{F}_{i}\}_{i\in\mathbb{N}}$
of $B_{n,P}$ is given by: 
\[
\mathcal{F}_{di+j}=k[x]s^{j}y^{i}+\mathcal{F}_{di+j-1}.
\]
where $i\in\mathbb{N}$ and $j\in\{0,\ldots d-1\}$. The associated
graded algebra is $\mathrm{Gr}_{\partial}(B_{n,P})=k[X,S,Y]/\langle X^{n}Y-S^{d}\rangle$,
and $B_{[di+j]}:=\mathcal{F}_{di+j}/\mathcal{F}_{di+j-1}=k[\overline{x}]\overline{s}^{j}\overline{y}^{i}$
where $i\in\mathbb{N}$ and $j\in\{0,\ldots d-1\}$ (see \S \ref{Sec:a more general case}
for more details). Corollary \ref{Coro:semi-rigidity- of-the-general-case}
below provides, in particular, an alternative argument formulated
in the $\mathrm{LND}$-filtration language proving directly that $\{\mathcal{F}_{i}\}_{i\in\mathbb{N}}$
is indeed the unique $\mathrm{LND}$-filtration of $B$.

\subsection{\label{Sec:alg.iso.semi-rigid} Algebraic isomorphisms between semi-rigid
$k$-domains}

\indent\newline\noindent  Let $\Psi:A\longrightarrow B$ be an algebraic
isomorphism between two $k$-domains. Given $\partial\in\mathrm{LND}(B)$,
then for any $n\in\mathbb{N}$ we have $(\Psi^{-1}\partial\Psi)^{n}=\Psi^{-1}\partial^{n}\Psi$.
So we see that $\partial_{\Psi}:=\Psi^{-1}\partial\Psi\in\mathrm{LND}(A)$.
An immediate consequence is that $\Psi(\mathrm{ML}(A))=\mathrm{ML}(B)$.
Furthermore, $\Psi\{\ker(\partial_{\Psi})\}=\ker(\partial)$, and
more generally, $\Psi$ sends elements of degree $n$, relative to
$\partial_{\Psi}$, to elements of the same degree $n$ relative to
$\partial$, that is, $\deg_{\partial_{\Psi}}(a)=\deg_{\partial}(\Psi(a))$
for all $a\in A$. So in particular, $\Psi(\mathrm{AL}(A))=\mathrm{AL}(B)$.
These properties, combined with Definition \ref{Def:semi-rigid},
give the following result.
\begin{prop}
\label{prop:iso-preserve-filtration} Let $\Psi:A\longrightarrow B$
be an isomorphism between two semi-rigid $k$-domains. Let $\{\mathcal{F}_{i}\}_{i\in\mathbb{N}}$\emph{
}$($resp. $\{\mathcal{G}_{i}\}_{i\in\mathbb{N}}$ \emph{$)$} be
the unique $\mathrm{LND}$-filtration\emph{ of $A$ $($resp. $B$}$)$.
Then: $\Psi(\mathcal{F}_{i})=\mathcal{G}_{i}$ for every $i$.
\end{prop}
In the case where $A=B$, we obtain an action of the group $\mathrm{Aut_{k}}(B)$
of algebraic $k$-automorphisms of $B$ by conjugation on $\mathrm{LND}(B)$.
As a consequence of Proposition \ref{prop:iso-preserve-filtration},
every $k$-automorphism of a semi-rigid $k$-domain $B$ preserves
its unique\emph{ }$\mathrm{LND}$-filtration $\{\mathcal{F}_{i}\}_{i\in\mathbb{N}}$.
Letting $\mathrm{Aut_{k}}\left(B,\mathrm{ML}(B)\right)$ be the sub-group
of $\mathrm{Aut_{k}}(B)$ consisting of elements whose induced action
on $\mathrm{ML}(B)$ is trivial, we have the following Corollary which
describes the structure of $\mathrm{Aut_{k}}(B)$. 
\begin{cor}
For every semi-rigid $k$-domain $B$, there exists an exact sequence
\[
0\rightarrow\mathrm{Aut_{k}}\left(B,\mathrm{ML}(B)\right)\rightarrow\mathrm{Aut_{k}}(B)\rightarrow\mathrm{Aut_{k}}\left(\mathrm{ML}(B)\right).
\]
Furthermore, every element of $\mathrm{Aut_{k}}\left(B,\mathrm{ML}(B)\right)$
induces for every $i\geq1$ an automorphism of $\mathcal{F}_{0}$-module
of each $\mathcal{F}_{i}$.
\end{cor}

\section{\textbf{A new class of semi-rigid }$k$-\textbf{domains}}

In this section, we introduce a new family of domains $R_{n,e,P,Q}$
of the form 
\[
R_{n,e,P,Q}:=k[X,Y,Z]/\left\langle X^{n}Y-P\left(X,Q(X,Y)-X^{e}Z\right)\right\rangle 
\]
 where $e\geq0$, $n\geq1$, $(n,e)\neq(1,0)$, $d,m\geq2$, 
\[
P(X,S)=S^{d}+f_{d-1}(X)S^{d-1}+\cdots+f_{1}(X)S+f_{0}(X)\,\,,\,\,\,\and
\]

\[
Q(X,Y)=Y^{m}+g_{m-1}(X)Y^{m-1}+\cdots+g_{1}(X)Y+g_{0}(X).
\]

The trivial case ($e=0$), corresponds to the Danielewski $k$-domains
$B_{n,P}=k[X,S,Y]/\left\langle X^{n}Y-P(X,S)\right\rangle $. Indeed,
the ring $R_{n,0,P,Q}=k[X,Y,Z]/\left\langle X^{n}Y-P\left(X,Q(X,Y)-Z\right)\right\rangle $
is isomorphic to $B_{n,P}$ via an isomorphism induced by $\Phi:k[X,S,Y]\longrightarrow k[X,Y,Z]$,
where $\Phi(X)=X$, $\Phi(S)=-Z+Q(X,Y)$, and $\Phi(Y)=Y$. It is
clear that $\Phi^{*}=\pi_{X^{n}Y-P\left(X,Q(X,Y)-Z\right)}\circ\Phi$
is surjective, where $\pi_{X^{n}Y-P\left(X,Q(X,Y)-Z\right)}:k[X,Y,Z]\rightarrow R_{n,e,P,Q}$
is the natural projection. Thus $R_{n,0,P,Q}=\mathrm{Im}\Phi^{*}\simeq k[X,S,Y]/\ker\Phi^{*}$.
This yields, in particular, that the ideal $\ker\Phi^{*}\subset k^{[3]}$
is principal. But since $\Phi^{*}(X^{n}Y-P(X,S))=0$, $\langle F\rangle\subset\ker\Phi^{*}$,
and $X^{n}Y-P(X,S)$ is irreducible, we conclude that $\langle X^{n}Y-P(X,S)\rangle=\ker\Phi^{*}$.
Therefore, $\Phi^{*}$ induces an isomorphism between the two rings.
\begin{rem}
Computing the $\mathrm{ML}$-invariant for these examples using known
techniques up to date is rather a hopeless task. Indeed, for the non-trivial
case of $R_{n,e,P,Q}$ where $e\neq0$, a real-valued weight degree
function $\omega$ on $k^{[3]}$ has to be of the form $\omega=(\frac{md-1}{nm-md+1}\lambda,\frac{n}{nm-md+1}\lambda,\lambda)$,
where $\lambda\in\mathbb{R}$, to induce a degree function $\omega_{0}$
on $R_{n,e,P,Q}$. Hence, the associated graded algebra, corresponding
to $\omega_{0}$-filtration, takes the form $\mathrm{Gr}_{\omega}(R_{n,e,P,Q})=k[X,Y,Z]/\langle X^{n}Y-(Y^{m}-X^{e}Z)^{d}\rangle$.
The latter ring is again another member of the new family that corresponds
to $R_{n,e,S^{d},Y^{m}}=\mathrm{Gr}_{\omega}(R_{n,e,P,Q})$. So any
hope of simplifying the study of locally nilpotent derivation of $R_{n,e,P,Q}$,
by studying the homogenous locally nilpotent derivation on the associated
graded algebra $R_{n,e,S^{d},Y^{m}}$, collapses. On the other hand,
the remaining choices of $\omega$ in $\mathbb{R}^{[3]}$ induces
a semi-degree function on $R_{n,e,P,Q}$ with the associated graded
algebra $\mathrm{Gr}_{\omega}(R_{n,e,P,Q})=k[X,Y,Z]/\langle(Y^{m}-X^{e}Z)^{d}\rangle$.
This is not an integral domain, which complicates the situation even
more. 

Nevertheless, the $\mathrm{LND}$-filtration method allows us to pass
through these complications as will be shown in the rest of this paper.
Indeed, consider $\omega\in\mathbb{N}^{[4]}$ the $\mathbb{N}$-weight
degree function on $k^{[4]}$ defined by $\omega(X,S,Y,Z)=(0,1,d,md)$,
then it induces $\omega_{R_{n,e,P,Q}}$ a degree function on 
\[
R_{n,e,P,Q}\simeq k[X,S,Y,Z]/\left\langle X^{n}Y-P\left(X,S\right),Q(X,Y)-X^{e}Z-S\right\rangle .
\]
It turns out that the degree function $\omega_{R_{n,e,P,Q}}$ coincides
with $\deg_{\partial}$ for any non-zero $\partial\in\mathrm{LND}(R_{n,e,P,Q})$.\\

\end{rem}

\subsection{Properties of the new class}

\indent\newline\noindent  Here, we point out some properties of $R_{n,e,P,Q}$
that we will establish in the rest of this section:

\subsubsection{\textbf{\emph{\large Algebraic construction:}}\emph{ }}

Consider the following $k$-domain 
\[
B_{n,P}=k[X,S,Y]/\langle X^{n}Y-P(X,S)\rangle,
\]
which is the Danielewski $k$-domain corresponding to the pair $(n,P)$.
Let us extend this ring by taking the sub-algebra of $k[X^{\pm1},S]$
generated by $B_{n,P}\subset k[X^{\pm1},S]$ and $z\in k[X^{\pm1},S]$,
where $z$ is an algebraic element over $k[X,S]$ that has a dependence
relation of the form 
\[
X^{nm+e}Z-[P(X,S)]^{m}-Xg_{m-1}(X)[P(X,S)]^{m-1}-\cdots-X^{m-1}g_{1}(X)P(X,S)-X^{m}g_{0}(X)+X^{nm}S.
\]

By sending $S$ to $Q(X,Y)-X^{e}Z$ we immediately see that: 
\[
B_{n,P}\subset R_{n,e,P,Q}\,\,,\,\,\and\,\,\,\, B_{nm+e,F}\subset R_{n,e,P,Q}\,\,,
\]
where $B_{nm+e,F}$ is the Danielewski $k$-domain corresponding to
the pair $(nm+e,F)$:
\[
B_{nm+e,F}=k[X,S,Z]/\langle X^{nm+e}Z-F(X,S)\rangle,
\]
and $F(X,S)=[P(X,S)]^{m}+Xg_{m-1}(X)[P(X,S)]^{m-1}+\cdots+X^{m-1}g_{1}(X)P(X,S)+X^{m}g_{0}(X)$. 

Clearly, we have $B_{n,P}.B_{nm+e,F}=R_{n,e,P,Q}$, which simply means
that $R_{n,e,P,Q}$ can be realized as the sub-algebra of $k[X^{\pm1},S]$
generated by both $B_{n,P}$ and $B_{nm+e,F}$. 

These new rings $R_{n,e,P,Q}$, for $e\neq0$, are not isomorphic
to any of Danielewski rings, see Proposition \ref{prop:not iso trivial case}.
Nevertheless, they share with them the property to come naturally
equipped with an irreducible locally nilpotent derivation. But in
contrast with the Danielewski rings, the corresponding derivation
on $k[X,Y,Z]$ are no longer triangular, in fact not even triangulable
by virtue of the characterization due to Daigle \cite{Daigle: traingulability}.
For instance: let $D$ be the locally nilpotent (triangular) derivation
of $k[X,S,Y,Z]$ defined by:

\begin{center}
$\partial(X)=0$, $\partial(S)=X^{n+e}$, $\partial(Y)=X^{e}\frac{\partial P}{\partial S}$,
and $\partial(Z)=\frac{\partial Q}{\partial Y}\,\frac{\partial P}{\partial S}-X^{n}$ 
\par\end{center}

\noindent  where $\frac{\partial P}{\partial S}=dS^{d-1}+(d-1)f_{d-1}(X)S^{d-2}+\cdots+f_{1}(X)$,
and $\frac{\partial Q}{\partial Y}=mY^{m-1}+(m-1)g_{m-1}(X)Y^{m-2}+\cdots+g_{1}(X)$.
Then $\partial$ induces a non-zero irreducible locally nilpotent
derivation of $R_{n,e,P,Q}$. Let $x,s,y,z$ be the class of $X,\, S:=Q(X,Y)-X^{e}Z,\, Y$,
and $Z$ in $R_{n,e,P,Q}$ , then: 
\[
\partial=x^{e}\frac{\partial P}{\partial s}\partial_{y}+(\frac{\partial Q}{\partial y}\,\frac{\partial P}{\partial s}-x^{n})\partial_{z}\in\mathrm{LND}(R_{n,e,P,Q})
\]

Furthermore, $\mathrm{ML}(R_{n,e,P,Q})=k[x]$ whenever $(n,e)\neq(1,0)$,
see Corollary \ref{Coro:semi-rigidity- of-the-general-case}. This
implies that $R_{n,e,P,Q}$ is semi-rigid, even almost rigid by virtue
of Proposition \ref{Pro:semi-rigid one dimensional}. Hence $\mathrm{\mathrm{AL}}(R_{n,e,P,Q})=k[X,S]$.
In addition, every non-zero locally nilpotent derivation of $R_{n,e,P,Q}$
restricts to a non-zero locally nilpotent derivation on $B_{n,P}$.
And most importantly, every $k$-automorphism of $R_{n,e,P,Q}$ restricts
to an automorphism of $B_{n,P}$. Also, it restricts to an $k$-automorphism
of $\mathrm{\mathrm{AL}}(R_{n,e,P,Q})$ (resp. $\mathrm{ML}(R_{n,e,P,Q})$).
So in particular, every $k$-automorphism of $B_{n,P}$ ($\simeq R_{n,0,P,Q}$)
restricts to a $k$-automorphism of $\mathrm{\mathrm{AL}}(R_{n,e,P})$
(resp. $\mathrm{ML}(R_{n,e,P,Q})$). But of course this is not the
full picture, see \ref{Sec:A-huge}.

\subsubsection{\textbf{\emph{\large Affine modification of the $\mathrm{\mathrm{AL}}$-invariant:}}\emph{ }}

Here, we present another point of view about the construction of the
new class.

\noindent  The affine modification of the $\mathrm{\mathrm{AL}}$-invariant
$\mathrm{\mathrm{AL}}(R_{n,e,P,Q})=k[X,S]$ along $X^{nm+e}$ with
center 
\[
I_{1}=\langle X^{nm+e},X^{n(m-1)+e}P(X,S),F(X,S)\rangle,
\]
see \cite[Definition 1.1]{Sh. Kaliman and M. Zaidenberg}, coincides
by virtue of \cite[Proposition 1.1]{Sh. Kaliman and M. Zaidenberg}
with 
\[
k[X,S]\left[I_{1}/X^{nm+e}\right]=k[X,S][P(X,S)/X^{n},F(X,S)/X^{nm+e}]=k[X,S,y,z]\simeq R_{n,e,P,Q}.
\]
Also, the affine modification of the $\mathrm{\mathrm{AL}}$-invariant
$\mathrm{\mathrm{AL}}(R_{n,e,P,Q})=k[X,S]$ along $X^{n}$ with center
$I_{2}=\langle X^{n},P(X,S)\rangle$ coincides with 
\[
k[X,S]\left[I_{2}/X^{n}\right]=k[X,S][P(X,S)/X^{n}]=k[X,S,y]\simeq B_{n,P}.
\]
 Finally, the affine modification of $B_{n,P}$ along $X^{e}$ with
center $I_{3}=\langle X^{e},Q(X,Y)-S\rangle$ coincides with 
\[
B_{n,P}\left[I_{3}/X^{e}\right]=B_{n,P}[\left(Q(X,Y)-S\right)/X^{e}])=B_{n,P}[z]\simeq R_{n,e,P,Q}.
\]

We put together previous observations in the following Proposition.
\begin{prop}
\label{prop:affine-mod} with the above notation we have:

\emph{$($1$)$} $R_{n,e,P,Q}$ is the affine modification of the
\textup{$\mathrm{\mathrm{AL}}$}-invariant along $X^{nm+e}$ with
center $I_{1}$.

\emph{$($2$)$} $B_{n,P}$ is the affine modification of the \textup{$\mathrm{\mathrm{AL}}$}-invariant
along $X^{n}$ with center $I_{2}$.

\emph{$($3$)$} $R_{n,e,P,Q}$ is the affine modification of $B_{n,P}$
along $X^{e}$ with center $I_{3}$.
\end{prop}

\subsubsection{\textbf{\emph{\large \label{Sec:A-huge} Invariant sub-algebras of
$R_{n,e,P,Q}$: }}}

\noindent  For simplicity let $Q(X,Y)=Y^{m}$. Denote $R_{n,e,P}:=R_{n,e,P,Y^{m}}$
and $B_{n,P}\simeq R_{n,0,P}$. Consider the two chains of inclusions:

The first chain of inclusions, which is realized by sending $S$ to
$Y^{m}-XZ$ for the first inclusion and by sending $Z$ to $XZ$ for
the rest steps 
\[
B_{n,P}\hookrightarrow R_{n,1,P}\hookrightarrow\cdots\hookrightarrow R_{n,e,P}.
\]

The second chain of inclusions, which is realized by sending $Y$
to $XY$ for every step 
\[
B_{1,P}\hookrightarrow B_{2,P}\hookrightarrow\cdots\hookrightarrow B_{n,P}.
\]

Together they produce the following chain of inclusions 
\[
B_{1,P}\hookrightarrow B_{2,P}\hookrightarrow\cdots\hookrightarrow B_{n,P}\hookrightarrow R_{n,1,P}\hookrightarrow\cdots\hookrightarrow R_{n,e,P}
\]
with the following properties.
\begin{thm}
With the above notation the following holds:

$($\emph{a$)$} Every non-zero $\partial\in\mathrm{LND}(R_{n,e,P})$
restricts to a non-zero $\mathrm{LND}(R_{n,e_{0},P})$ for any $e_{0}\in\{1,\ldots,e\}$.
Also, it restricts to a non-zero $\mathrm{LND}(B_{n_{0},P})$ for
any $n_{0}\in\{0,\ldots,n\}$. 

\emph{$($b$)$} Every algebraic $k$-automorphism of $R_{n,e,P}$
restricts to an algebraic $k$-automorphism of $R_{n,e_{0},P}$ for
any $e_{0}\in\{1,\ldots,e\}$. Also, it restricts to a $k$-automorphism
of $B_{n_{0},P}$ for any $n_{0}\in\{1,\ldots,n\}$.

\emph{$($c$)$} $R_{n_{1},e_{1},P}$ $\simeq$ $R_{n_{2},e_{2},P}$
if an only if $n_{1}=n_{2}$ and $e_{1}=e_{2}$. Hence these $k$-domains
are not algebraically isomorphic to each other (pairwise). 

\emph{$($d$)$} Every element of the set 
\[
\left\{ \mathrm{\mathrm{ML}}(R_{n,e,P})=k[X],\,\mathrm{\mathrm{AL}}(R_{n,e,P})=k[X,S],\, B_{n_{0},P}\,,\, R_{n,e_{0},P};\,\, n_{0}\in\{1,\ldots,n\},e_{0}\in\{1,\ldots,e\}\right\} 
\]
represents an invariant sub-algebra of $R_{n,e,P}$.

\emph{$($e$)$} \textup{$\mathrm{\mathrm{AL}}_{0}(R_{n,e,P})=\mathcal{D}(B)=\mathrm{\mathrm{ML}}(R_{n,e,P})\hookrightarrow\mathrm{\mathrm{AL}}(R_{n,e,P})=\mathrm{\mathrm{AL}}_{1}(R_{n,e,P})=\cdots=\mathrm{\mathrm{AL}}_{d-1}(R_{n,e,P})\hookrightarrow B_{n,P}=\mathrm{\mathrm{AL}}_{d}(R_{n,e,P})=\cdots=\mathrm{\mathrm{AL}}_{md-1}(R_{n,e,P})\hookrightarrow\mathrm{\mathrm{AL}}_{md}(R_{n,e,P})=R_{n,e,P}$.}\end{thm}
\begin{proof}
\emph{$($}a$)$\emph{ }Is immediate by Corollary \ref{Coro:semi-rigidity- of-the-general-case}
below. $($b$)$ is an immediate consequence of Theorem \ref{Theo:auto.new example }
below. $($c$)$ a consequence of Proposition \ref{prop:iso non-tivial case}
and Proposition \ref{prop:not iso trivial case} below. $($a$)$,
$($b$)$, and Corollary \ref{Coro:semi-rigidity- of-the-general-case}
implies $($d$)$. Finally, $($e$)$ is a trivial consequence of
Theorem \ref{Thm:theLND-filtration}, and Definition \ref{Def:The-AL-iinvariant}.
\end{proof}

\subsection{\label{Sec:A toy example} A toy example}

\indent\newline\noindent  We will begin with a very elementary example
illustrating the steps needed to determine the $\mathrm{LND}$-filtration
and its associated graded algebra, and then we proceed to the general
case. We let 
\[
R=k[X,Y,Z]/\langle X^{2}Y-(Y^{2}-XZ)^{2}\rangle
\]
 and we let $x$, $y$, $z$ be the class of $X$, $Y$, and $Z$
in $R$. A direct computation reveals that the derivation 
\[
2XS\partial_{Y}+(4YS-X^{2})\partial_{Z}
\]
of $k[X,Y,Z]$ where $S:=Y^{2}-XZ$ is locally nilpotent and annihilates
the polynomial $X^{2}Y-(Y^{2}-XZ)^{2}$. Therefore, it induces a locally
nilpotent derivation $\partial$ of $R$ for which we have $\partial(x)=0$,
$\partial^{3}(y)=0,\partial^{5}(z)=0$. Furthermore, the element $s=y^{2}-xz$
is a local slice for $\partial$ with $\partial(s)=x^{3}$. So we
have $\deg_{\partial}(x)=0$, $\deg_{\partial}(y)=2$, $\deg_{\partial}(z)=4$,
$\deg_{\partial}(s)=1$. The kernel of $\partial$ is $k[x]$ and
the plinth ideal is the principal ideal generated by $x^{3}$.
\begin{prop}
\label{Prop:filtration for a toy example} With the notation above,
we have:

\emph{$($1$)$ }The $\partial$-filtration $\{\mathcal{F}_{i}\}_{i\in\mathbb{N}}$
is given by : 
\[
\mathcal{F}_{4i+2j+l}=k[x]s^{l}y^{j}z^{i}+\mathcal{F}_{4i+2j+l-1}
\]
where $i\in\mathbb{N}$, $j\in\{0,1\}$, $l\in\{0,1\}$.

\emph{$($2$)$} The associated graded algebra $\mathrm{Gr_{\partial}}(R)=\oplus_{i\in\mathbb{N}}R_{[i]}$,
where $R_{[i]}=\mathcal{F}_{i}/\mathcal{F}_{i-1}$, is generated by
$\overline{x}=\mathrm{gr}_{\partial}(x)$, $\overline{y}=\mathrm{gr}_{\partial}(y)$,
$\overline{z}=\mathrm{gr}_{\partial}(z)$, $\overline{s}=\mathrm{gr}_{\partial}(s)$
as an algebra over $k$ with relations $\overline{x}^{2}\overline{y}=\overline{s}^{2}$
and $\overline{x}\,\overline{z}=\overline{y}^{2}$, that is $\mathrm{Gr_{\partial}}(R)\simeq k[X,Y,Z,S]/\langle X^{2}Y-S^{2},X\, Z-Y^{2}\rangle$.
Furthermore:
\[
R_{[4i+2j+l]}=k[\overline{x}]\overline{s}^{l}\overline{y}^{j}\overline{z}^{i}
\]
where $i\in\mathbb{N}$, $j\in\{0,1\}$, $l\in\{0,1,2,3\}$.\end{prop}
\begin{proof}
1) First, the $\partial$-filtration \textit{$\{\mathcal{F}_{i}\}_{i\in\mathbb{N}}$}
is given by $\mathcal{F}_{r}=\sum_{h\leq r}H_{h}$ where $H_{h}:=\sum_{u+2v+4w=h}k[x]\left(s^{u}y^{v}z^{w}\right)$
and $u,v,w,h\in\mathbb{N}$. To show this, let $J$ be the ideal in
$k^{[4]}=k[X,Y,Z,S]$ defined by 
\[
J=\langle X^{2}Y-S^{2},Y^{2}-XZ-S\rangle.
\]
Define an $\mathbb{N}$-weight degree function $\omega$ on $k^{[4]}$
by declaring that $\omega(X)=0$, $\omega(S)=1$, $\omega(Y)=2$,
and $\omega(Z)=4$. By Proposition \ref{Pro:when-it-is-proper}, the
$\mathbb{N}$-filtration \textit{$\{\mathcal{G}_{r}\}_{i\in\mathbb{N}}$}
where $\mathcal{G}_{r}=\sum_{h\leq r}H_{h}$ is proper if and only
if $\hat{J}$ is prime. Which is the case since $\hat{J}=\left\langle X^{2}Y-S^{2},Y^{2}-XZ\right\rangle $
is prime. Thus by Lemma \ref{lem:equal-if-it-is-proper} we get the
desired description.

Second, let \textit{$l\in\{0,1\}$ }\textit{\emph{and }}\textit{$j\in\{0,1,2,3\}$
}\textit{\emph{be such that }} $l:=r$ mod $2$, $j:=r-l$ mod $4$,
and $i:=\frac{r-2j-l}{4}$. Then we get the following unique expression
$r=4i+2j+l$. Since $\mathcal{F}_{r}=\sum_{u+2v+4w=r}k[x]\left(s^{u}y^{v}z^{w}\right)+\mathcal{F}_{r-1}$,
we conclude in particular that $\mathcal{F}_{r}\supseteq k[x]s^{l}y^{j}z^{i}+\mathcal{F}_{r-1}$.
For the other inclusion, the relation $x^{2}y=s^{2}$ allows to write
$s^{u}y^{v}z^{w}=x^{e}s^{l}y^{v_{0}}z^{w}$ and from the relation
$y^{2}=s+zx$ we get $x^{e}s^{l}y^{v_{0}}z^{w}=x^{e}s^{l}y^{j}(s+xz)^{n}z^{w}$.
Since the monomial with the highest degree relative to $\deg_{\partial}$
in $(s+xz)^{n}$ is $x^{n}.z^{n}$, we deduce that $x^{e}s^{l}y^{j}(s+xz)^{n}z^{w}=x^{e+n}s^{l}y^{j}z^{w+n}+\sum M_{\beta}$
where $M_{\beta}$ is monomial in $x$, $y$, $s$, $z$ of degree
less than $r$. Since the expression $r=4i+2j+l$ is unique, we get
$w+n=i$. So $s^{u}y^{v}z^{w}=x^{e+n}s^{l}y^{j}z^{i}+f$ where $f\in\mathcal{F}_{r-1}$.
Thus $k[x]\left(s^{u}y^{v}z^{w}\right)\subseteq k[x]s^{l}y^{j}z^{i}+\mathcal{F}_{r-1}$
and finally $\mathcal{F}_{r}=k[x]s^{l}y^{j}z^{i}+\mathcal{F}_{r-1}$.

2) By part (1), an element $f$ of degree $r$ can be written as $f=g(x)s^{l}y^{j}z^{i}+f_{0}$
where $f_{0}\in\mathcal{F}_{r-1}$, $l=r$ mod $2$, $j=r-l$ mod
$4$, $i=\frac{r-2j-l}{4}$, and $i\in\mathbb{N}$, $j\in\{0,1\}$,
$l\in\{0,1\}$. So by Lemma \ref{lem:graded relation}, P2, P1, and
P3, respectively we get 
\[
\overline{f}=\overline{g(x)s^{l}y^{j}z^{i}+h}=\overline{g(x)s^{l}y^{j}z^{i}}=\overline{g(x)}\overline{s}^{l}\overline{y}^{j}\overline{z}^{i}=g(\overline{x})\overline{s}^{l}\overline{y}^{j}\overline{z}^{i}
\]
 and therefore $\overline{B}_{[4i+2j+l]}=k[\overline{x}]\overline{s}^{l}\overline{y}^{j}\overline{z}^{i}$. 

Finally, by Proposition \ref{prop:graded algebra}, $\mathrm{Gr}_{\partial}(B)=k[X,Y,Z,S]/\langle X^{2}Y-S^{2},X\, Z-Y^{2}\rangle$.
\end{proof}

\subsection{\label{Sec:a more general case} The general case}

\indent\newline\noindent  We now consider more generally rings $R_{n,e,P,Q}$
 of the form 
\[
k[X,Y,Z]/\left\langle X^{n}Y-P\left(X,Q(X,Y)-X^{e}Z\right)\right\rangle 
\]
 where $e\geq0$, $n\geq1$, $(n,e)\neq(1,0)$, $d,m\geq2$, 
\[
P(X,S)=S^{d}+f_{d-1}(X)S^{d-1}+\cdots+f_{1}(X)S+f_{0}(X)\,\,,\,\,\,\,\and
\]
\[
Q(X,Y)=Y^{m}+g_{m-1}(X)Y^{m-1}+\cdots+g_{1}(X)Y+g_{0}(X)
\]

Up to a change of variable of the form $Y\mapsto Y-c$ where $c\in k$,
we may assume that $0\in\mathrm{Spec}(R_{n,e,P,Q})$. Let $x$, $y$,
$z$ be the class of $X$, $Y$, and $Z$ in $R_{n,e,P,Q}$. Define
$\partial$ by: $\partial(x)=0$, $\partial(s)=x^{n+e}$ where $s:=Q(x,y)-x^{e}z$.
Considering the relation $x^{n}y=P(x,Q(x,y)-x^{e}z)$, a simple computation
leads to $\partial(y)=x^{e}\frac{\partial P}{\partial s}$ ,$\partial(z)=\frac{\partial Q}{\partial y}\frac{\partial P}{\partial s}-x^{n}$,
that is 
\[
\partial:=x^{e}\frac{\partial P}{\partial s}\partial_{y}+(\frac{\partial Q}{\partial y}\,\frac{\partial P}{\partial s}-x^{n})\partial_{z}
\]
where $\frac{\partial P}{\partial s}=ds^{d-1}+(d-1)f_{d-1}(x)s^{d-2}+\cdots+f_{1}(x)$,
and $\frac{\partial Q}{\partial y}=my^{m-1}+(m-1)g_{m-1}(x)y^{m-2}+\cdots+g_{1}(x)$.
Since $\partial(x^{n}y-P(x,Q(x,y)-x^{e}z))=0$ and $\partial^{d+1}(y)=0,\partial^{md+1}(z)=0$,
$\partial$ is a well-defined locally nilpotent derivation of \textbf{$R_{n,e,P,Q}$}.
The element $s$ is a local slice for $\partial$ by construction,
and a direct computation shows that $\deg_{\partial}(x)=0$, $\deg_{\partial}(y)=d$,
$\deg_{\partial}(z)=md$ and $\deg_{\partial}(s)=1$. The kernel of
$\partial$ is equal to $k[x]$. One checks further that the plinth
ideal is equal to $\mathrm{pl}(\partial)=\left\langle x^{n+e}\right\rangle $.
Furthermore, the exact same method as in the proof of Proposition
\ref{Prop:filtration for a toy example} provides a full description
of the $\partial$-filtration and its associated graded algebra, that
is:
\begin{thm}
\label{Thm:theLND-filtration} Let $\partial$ be defined as above,
then we have:

\emph{$($1$)$} The $\partial$-filtration $\mathcal{F}=\{\mathcal{F}_{i}\}_{i\in\mathbb{N}}$
is given by : 
\[
\mathcal{F}_{mdi+dj+l}=k[x]s^{l}y^{j}z^{i}+\mathcal{F}_{mdi+dj+l-1}
\]
where $i\in\mathbb{N}$, $j\in\{0,\ldots,m-1\}$, $l\in\{0,\ldots,d-1\}$.

\emph{$($2$)$ }The associated graded algebra $\mathrm{Gr}(R_{n,e,P,Q})=\oplus_{i\in\mathbb{N}}R_{[i]}$,
where $R_{[i]}=\mathcal{F}_{i}/\mathcal{F}_{i-1}$, is generated by
$\overline{x}=\mathrm{gr}_{\partial}(x)$, $\overline{y}=\mathrm{gr}_{\partial}(y)$,
$\overline{z}=\mathrm{gr}_{\partial}(z)$, $\overline{s}=\mathrm{gr}_{\partial}(s)$
as an algebra over $k$ with relations $\overline{x}^{n}\overline{y}=\overline{s}^{d}$
and $\overline{x}^{e}\,\overline{z}=\overline{y}^{m}$ , that is,
$\mathrm{Gr}(R_{n,e,P,Q})=k[X,Y,Z,S]/\langle X^{n}Y-S^{d},X^{e}Z-Y^{m}\rangle$.
In addition, we have: 
\[
R_{[mdi+dj+l]}=k[\overline{x}]\overline{s}^{l}\overline{y}^{j}\overline{z}^{i}
\]
where $i\in\mathbb{N}$, $j\in\{0,\ldots,m-1\}$, $l\in\{0,\ldots,d-1\}$.\end{thm}
\begin{proof}
Consider $\omega\in\mathbb{N}^{[4]}$ the $\mathbb{N}$-weight degree
function on $k^{[4]}=k[X,S,Y,Z]$ defined by $\omega(X,S,Y,Z)=(0,1,d,md)$.
Let $J=\langle X^{n}Y-P(X,S),Q(X,Y)-X^{e}Z-S\rangle$, it is clear
that we can identify $R_{n,e,P,Q}$ with $R_{n,e,P,Q}\simeq k[X,S,Y,Z]/\left\langle X^{n}Y-P\left(X,S\right),Q(X,Y)-X^{e}Z-S\right\rangle $.
Since $\hat{J}=\langle X^{n}Y-S^{d},Y^{m}-X^{e}Z\rangle$ and $\langle X^{n}Y-S^{d},Y^{m}-X^{e}Z\rangle$
is a prime ideal in $k^{[4]}$, $\omega$ induces $\omega_{R_{n,e,P,Q}}$
(defined as in Proposition \ref{Pro:when-it-is-proper}) a degree
function on $k[X,S,Y,Z]/\left\langle X^{n}Y-P\left(X,S\right),Q(X,Y)-X^{e}Z-S\right\rangle $.
The latter coincides with $\deg_{\partial}$ by virtue of Lemma \ref{lem:equal-if-it-is-proper}.
Hence by Proposition \ref{prop:graded algebra}, the associated graded
algebra is given by $\mathrm{Gr}(R_{n,e,P,Q})=k[X,Y,Z,S]/\langle X^{n}Y-S^{d},X^{e}Z-Y^{m}\rangle$.
The explicit description of $\mathcal{F}_{mdi+dj+l}$ and $R_{[mdi+dj+l]}$,
using the exact same method as in the proof of Proposition \ref{Prop:filtration for a toy example},
is left to the reader.\end{proof}
\begin{cor}
\label{Coro:semi-rigidity- of-the-general-case} With the above notation,
the following hold:

\emph{$($1$)$ }$\mathrm{ML}(R_{n,e,P,Q})=k[x]$. Consequently, $R_{n,e,P,Q}$
is semi-rigid, and its unique $\mathrm{LND}$-filtration is the $\partial$-filtration. 

\emph{$($2$)$} Every $D\in\mathrm{LND}(R_{n,e,P,Q})$ has the form
$D=f(x)\partial$. Consequently, $R_{n,e,P,Q}$ is almost-rigid.\end{cor}
\begin{proof}
(1) Given a non-zero $D\in\mathrm{LND}(R_{n,e,P,Q})$. By Proposition
\ref{Pro:Daigle}, $D$ respects the $\partial$-filtration and induces
a non-zero locally nilpotent derivation $\overline{D}$ of $\mathrm{Gr}(R_{n,e,P,Q})$.
Suppose that $f\in\ker(D)\setminus k$, then $\overline{f}\in\ker(\overline{D})\setminus k$
is an homogenous element of $\mathrm{Gr}(R_{n,e,P,Q})$. So there
exists $i\in\mathbb{N}$ such that $\overline{f}\in R_{[i]}$.

Assume that $\overline{f}\notin k[\overline{x}]=R_{[0]}$, then one
of the elements $\overline{s}$, $\overline{y}$, $\overline{z}$
must divides $\overline{f}$ by Theorem \ref{Thm:theLND-filtration}.
This leads to a contradiction as follows: 

If $\overline{s}$ divides $\overline{f}$ , then $\overline{s}\in\ker(\overline{D})$
as $\ker(\overline{D})$ is factorially closed, and for the same reason
$\overline{x},\overline{y}\in\ker(\overline{D})$ due to the relation
$\overline{x}^{n}\overline{y}=\overline{s}^{d}$. Then by the relation
$\overline{x}^{e}\,\overline{z}=\overline{y}^{m}$, we must have $\overline{z}\in\ker(\overline{D})$,
which means $\overline{D}=0$, a contradiction. In the same way, we
get a contradiction if $\overline{y}$ divides $\overline{f}$. 

Finally, if $\overline{z}$ divides $\overline{f}$, then $\overline{D}(\overline{z})=0$.
So $\overline{D}$ induces in a natural way a locally nilpotent derivation
$\widetilde{D}$ of the ring $\widetilde{R}=k(Z)[X,Y,S]/\langle X^{n}\, Y-S^{d},X^{e}\, Z-Y^{m}\rangle$.
It follows from the Jacobian criterion that $0\in\mathrm{Spec(\widetilde{R})}$
is a singular point, therefore $\widetilde{R}$ is rigid, see \cite[Corollary 1.29]{key-gene}.
Hence $\widetilde{D}=0$, which implies $\overline{D}=0$, a contradiction.

So the only possibility is that $\overline{f}\in k[\overline{x}]$,
and this means $\deg_{\partial}(f)=0$, thus $f\in k[x]$ and $\ker(D)\subset k[x]$.
Finally, $k[x]=\ker(D)$ because $\mathrm{tr.deg_{k}}(\ker(D))=1$
and $k[x]$ is algebraically closed in $R_{n,e,P,Q}$. So $\mathrm{ML}(R_{n,e,P,Q})=k[x]$.

(2) follows immediately from Proposition \ref{Pro:semi-rigid one dimensional}.
\end{proof}
As a direct consequence of Theorem \ref{Coro:semi-rigidity- of-the-general-case},
we have the following two Corollary.
\begin{cor}
\textup{\emph{\label{cor:The ALi-invariant for the new class} The
$\mathrm{\mathrm{AL}}_{i}$-invariant of $R_{n,e,P,Q}$ are given
by:}}

\textup{$($1$)$}\textup{\emph{ $\mathrm{\mathrm{AL}}_{0}(R_{n,e,P,Q})=\mathcal{D}(B)=\mathrm{\mathrm{ML}}(R_{n,e,P,Q})=k[x]$.}}

\textup{$($2$)$}\textup{\emph{ $\mathrm{\mathrm{AL}}(R_{n,e,P,Q})=\mathrm{\mathrm{AL}}_{1}(R_{n,e,P,Q})=\cdots=\mathrm{\mathrm{AL}}_{d-1}(R_{n,e,P,Q})=k[x,s]$.}}

\textup{$($3$)$}\textup{\emph{ $B_{n,P}=\mathrm{\mathrm{AL}}_{d}(R_{n,e,P,Q})=\cdots=\mathrm{\mathrm{AL}}_{md-1}(R_{n,e,P,Q})=k[x,s,y]\simeq B_{n,P}$.}}

\textup{$($4$)$}\textup{\emph{ $\mathrm{\mathrm{AL}}_{md}(R_{n,e,P,Q})=R_{n,e,P,Q}$.}}
\end{cor}
Also, we have an interesting fact. Consider the following chain of
inclusions, realized by the identity for the first inclusion and by
sending $S$ to $Q(X,Y)-X^{e}Z$ for the second one. 
\[
\mathrm{\mathrm{AL}}(R_{n,e,P,Q})=k[X,S]\hookrightarrow B_{n,P}=k[X,S,Y]/\left\langle X^{n}Y-P\left(X,S\right)\right\rangle \hookrightarrow R_{n,e,P,Q}.
\]
Then every non-zero locally nilpotent derivation of $R_{n,e,P,Q}$
restricts to a non-zero locally nilpotent derivation of $k[x,s,y]\simeq B_{n,P}$
($\simeq R_{n,0,P,Q}$). Also, it restricts to a non-zero locally
nilpotent derivation of the sub-algebra $k[x,s]\simeq k^{[2]}$.
\begin{cor}
\label{cor:LND restricts} Every non-zero $D\in\mathrm{LND}(R_{n,e,P,Q})$
restricts to a non-zero locally nilpotent derivation of $B_{n,P}$
$($resp. $\mathrm{\mathrm{AL}}(R_{n,e,P,Q})$$)$.\\

\end{cor}

\subsection{{\Large $\mathrm{Aut_{k}}$} for the new class \ref{Sec:a more general case}}

\indent\newline\noindent  For simplicity we only deal with the case
where $Q(X,Y)=Y^{m}$. Up to change of variable of the form $S$ by
$S-\frac{f_{d-1}(X)}{d}$ we may assume without loss of generality
that $f_{d-1}(X)=0$. 

Let $R_{n,e,P}$ denote the ring 
\[
R_{n,e,P}:=R_{n,e,P,Y^{m}}=k[X,Y,Z]/\left\langle X^{n}Y-P(X,Y^{m}-X^{e}Z)\right\rangle =k[x,y,z]
\]
where $P(X,S)=S^{d}+f_{d-2}(X)S^{d-2}+\cdots+f_{1}(X)S+f_{0}(X)$,
$e\geq0$, $n\geq1$, $(n,e)\neq(1,0)$, and \textit{$d,m\geq2$.
}Let $\mathcal{F}=\{\mathcal{F}_{i}\}_{i\in\mathbb{N}}$ be its unique
$\mathrm{LND}$-filtration. 

As an immediate consequence of Corollary \ref{cor:The ALi-invariant for the new class},
we have the following Corollary that shows how the computation of
the\emph{ }\textit{\emph{algebraic $k$-automorphism group of }}$R_{n,e,P}$
\textit{\emph{can be simplified by }}consider the following chain
of inclusions, realized by sending $S$ to $Y^{m}-X^{e}Z$ for the
last one. 
\[
\mathrm{ML}(R_{n,e,P})=k[x]\hookrightarrow\mathrm{\mathrm{AL}}(R_{n,e,P})=k[x,s]\hookrightarrow\mathrm{\mathrm{AL}}_{d}(R_{n,e,P})=B_{n,P}\hookrightarrow R_{n,e,P}.
\]

\begin{cor}
\label{cor:aut.restricts} Every algebraic $k$-automorphism of $R_{n,e,P}$
restricts to: 

\emph{$($1$)$} a $k$-automorphism of $\mathrm{\mathrm{AL}}_{d}(R_{n,e,P})=B_{n,P}$
\emph{$($}$\simeq R_{n,0,P}$$)$.

\emph{$($2$)$} a $k$-automorphism of $\mathrm{\mathrm{AL}}(R_{n,e,P})=k[x,s]$.

\emph{$($3$)$} a $k$-automorphism of $\mathrm{ML}(R_{n,e,P})=k[x]$.
\end{cor}
\textit{\emph{Nevertheless, for those who are not familiar with }}the\emph{
}\textit{\emph{algebraic $k$-automorphism group of $B_{n,P}$, we
present a complete proof for the next Theorem \ref{Theo:auto.new example }
without implicitly using the previous Corollary \ref{cor:aut.restricts}.
Let $\lambda,\mu\in k^{*}$}}\textit{,} and $a(x)\in k[x]$\textit{\emph{.
Denote }}$s=y^{m}-x^{e}z$ and\textit{\emph{ }}\textit{$W:=\frac{P(\lambda x,\mu s+x^{n+e}a(x))-\mu^{d}P(x,s)}{\lambda^{n}x^{n}}.$}
\begin{thm}
\label{Theo:auto.new example } Every algebraic $k$-automorphism
$\alpha$ of $R_{n,e,P}$ has the form: 
\[
\alpha(x,s,y,z)=(\lambda x,\,\mu s+x^{n+e}a(x),\,\frac{\mu^{d}}{\lambda^{n}}y+W,\,\frac{\mu^{dm}}{\lambda^{nm+e}}z+\frac{(\frac{\mu^{d}}{\lambda^{n}}y+W)^{m}-\frac{\mu^{dm}}{\lambda^{nm}}y^{m}+x^{n+e}a(x)}{\lambda^{e}x^{e}})
\]
where $\lambda,\mu\in k^{*}$ verify both: $\frac{\mu^{dm-1}}{\lambda^{nm}}=1$
and $f_{d-i}(\lambda x)\equiv\mu^{i}f_{d-i}(x)\mod\, x^{n+e}$ for
every $i\in\{2,\ldots,d\}$. \end{thm}
\begin{proof}
By Proposition \ref{prop:iso-preserve-filtration}, $\alpha$ preserve
the unique $\mathrm{LND}$-filtration of $B$, described in Theorem
\ref{Thm:theLND-filtration}.

Thus we must have $\alpha(x)\in\mathcal{F}_{0}=k[x]$, $\alpha(s)\in\mathcal{F}_{1}=k[x]s+k[x]$,
$\alpha(y)\in\mathcal{F}_{d}=k[x]y+\mathcal{F}_{d-1}$ and $\alpha(z)\in\mathcal{F}_{md}=k[x].z+\mathcal{F}_{md-1}$.
In addition, $\alpha$ restricts to an automorphism of $\mathcal{F}_{0}=k[x]$.
Therefore, $\alpha(x)=\lambda x+c$ where $\lambda\in k^{*}$, and
$c\in k$.

Since $\alpha$ is invertible we get $\alpha(s)=\mu s+b(x)$, $\alpha(y)=\varepsilon y+h(x,s)$,
and $\alpha(z)=\xi z+g(x,s,y)$ for some $\mu,\varepsilon,\xi\in k^{*}$,
$b(x)\in k[x]$, $h(x,s)\in k[x,s]$, and $g(x,s,y)\in k[x,s,y]$. 

By Corollary \ref{Coro:semi-rigidity- of-the-general-case} (2) every
$D\in\mathrm{LND}(R_{n,e,P})$ has the form $D=f(x)\partial$. In
particular, $\partial_{\alpha}:=\alpha^{-1}\partial\alpha=f(x)\partial$
for some $f(x)\in k[x]$. Since $\alpha\partial_{\alpha}=\partial\alpha$
we have $\partial(\alpha(s))=\alpha(f(x)\partial(s))=f(\alpha(x))\alpha(x^{n+e})$
where ($\partial(s)=x^{n+e}$). So we get $\partial(\mu s+b(x))=f(\alpha(x))\left(\lambda x+c\right)^{n+e}$.
Since $\partial(\mu y+b(x))=\mu x^{n+e}$, $x$ divides $(\lambda x+c)^{n+e}$
in $k[x]$, and this is possible only if $c=0$, so we finally get
$\alpha(x)=\lambda x$.

Applying $\alpha$ to the relation $x^{n}y=P(x,s)$ in $R_{n,e,P}$,
we get $\lambda^{n}x^{n}\alpha(y)=P(\lambda x,\mu s+b(x))=\mu^{d}P(x,s)+d\mu^{d-1}s^{d-1}b(x)+H(x,s)$
where $\deg_{s}H\leq d-2$. Since $x^{n}$ divides both $x^{n}y$
and $P(x,s)$, and $\deg_{s}H\leq m-2$, we conclude that $x^{n}$
divides $d\mu^{d-1}s^{d-1}b(x)+H(x,s)$ in $k[x,s]$. So $x^{n}$
divides $b(x)$, that is $\alpha(s)=\mu s+x^{n}a(x)$. 

In addition, $x^{n}$ divides every coefficient of $H$ as a polynomial
in $s$, so $x^{n}$ divides $-\mu^{d}f_{d-i}(x)+\mu^{d-i}f_{d-i}(\lambda x)$
because coefficients of $H(s)$ are of the form $q(x,s)b(x)-\mu^{d}f_{d-i}(x)+\mu^{d-i}f_{d-i}(\lambda x)$.
Since $x^{n}$ divides $b(x)$, it divides also $-\mu^{i}f_{d-i}(x)+f_{d-i}(\lambda x)$
for every $i$.

Now $\alpha(x)$ and $\alpha(s)$ fully determine $\alpha(y)$: 
\[
\alpha(y)=\frac{\mu^{d}}{\lambda^{n}}y+\frac{P(\lambda x,\mu s+x^{n}a(x))-\mu^{d}P(x,s)}{\lambda^{n}x^{n}}.
\]

Apply $\alpha$ to $x^{e}z=y^{m}-s$ to get $\lambda^{e}x^{e}\alpha(z)=(\frac{\mu^{h}}{\lambda^{n}}y+W)^{m}-\mu s-x^{n}a(x)$
where $W=\frac{P(\lambda x,\mu s+x^{n}a(x))-\mu^{n}P(x,s)}{\lambda^{n}x^{n}}\in k[x,s]$.
So we have $\lambda^{e}x^{e}\alpha(z)=[\frac{\mu^{nm}}{\lambda^{nm}}y^{m}-\frac{\mu^{nm}}{\lambda^{nm}}s]+(\frac{\mu^{nm}}{\lambda^{nm}}-\mu)s+m(\frac{\mu^{n}}{\lambda^{n}}y)^{m-1}W+\ldots+W^{m}-x^{n}a(x)$.
Since $\frac{\mu^{nm}}{\lambda^{nm}}y^{m}-\frac{\mu^{nm}}{\lambda^{nm}}s=\frac{\mu^{nm}}{\lambda^{nm}}x^{e}z$,
we see that $x^{e}$ divides $G:=(\frac{\mu^{nm}}{\lambda^{nm}}-\mu)s+m(\frac{\mu^{n}}{\lambda^{n}}y)^{m-1}W+\ldots+W^{m}-x^{n}a(x)$
in $k[x,s,y]\subset R_{n,e,P}$ because $\deg_{s}G\leq\deg_{s}(z)-1=md-1$. 

Note that $W=\frac{d\mu^{d-1}s^{d-1}x^{n}a(x)+H(x,s)}{\lambda^{n}x^{n}}$,
thus $\deg_{s}W<d$. So by applying the map $\mathrm{gr}_{\mathcal{F}}$
we get
\[
\overline{G}=\overline{m(\frac{\mu^{n}}{\lambda^{n}}y)^{m-1}W}=\overline{m(\frac{\mu^{n}}{\lambda^{n}}y)^{m-1}}\,\overline{d\frac{\mu^{d-1}}{\lambda^{n}}s^{d-1}a(x)}.
\]

Since $x^{e}$ divides $G$, $x^{e}$ divides $a(x)$. Thus $x^{e}$
divides every coefficients of $G$ as a polynomial in $y$. So $x^{e}$
divides $W$ and $\frac{\mu^{nm}}{\lambda^{nm}}-\mu=0$. This means
that $x^{n+e}$ divides $f_{d-i}(\lambda x)-\mu^{i}f_{d-i}(x)$ for
all $i\in\{2,\ldots,d\}$, and $\frac{\mu^{nm}}{\lambda^{nm}}=\mu$.
Finally, by the relation \textit{$s=y^{m}-x^{e}z$,} we get $\alpha(z)=\frac{\mu^{nm}}{\lambda^{nm+e}}z+\frac{m(\frac{\mu^{n}}{\lambda^{n}}y)^{m-1}W+\ldots+W^{m}+x^{n}a(x)}{\lambda^{e}x^{e}}$,
and we are done.
\end{proof}
The next Corollary describes the algebraic $k$-automorphism group
of $R_{n,e,P}$ in terms of the algebraic $k$-automorphism group
of the $\mathrm{\mathrm{AL}}$-invariant. Denote by $\mathcal{A}_{1}$
the sub-group of $\mathrm{\mathrm{Aut}}_{k}\left(\mathrm{\mathrm{AL}}(R_{n,e,P})\right)=\mathrm{\mathrm{Aut}}_{k}\left(k[X,S]\right)$
of automorphisms which preserve the ideals $\langle X\rangle$ and
{\small $I_{1}=\langle X^{nm+e},X^{n(m-1)+e}P(X,S),F(X,S)\rangle$}.
Also, denote by $\mathcal{A}_{2}$ the sub-group of $\mathrm{\mathrm{Aut}}_{k}\left(\mathrm{\mathrm{AL}}(R_{n,e,P})\right)$
of automorphisms which preserve the ideals $\langle X\rangle$ and
$I_{2}=\langle X^{n},P(X,S)\rangle$. Finally, denote by $\mathcal{A}_{3}$
the sub-group of $\mathrm{\mathrm{Aut}}_{k}\left(B_{n,P}\right)$
of automorphisms which preserve the ideals $\langle X\rangle$ and
$I_{3}=\langle X^{e},Q(X,Y)-S\rangle$. Then,
\begin{cor}
In the case $e\neq0$, $\mathrm{\mathrm{Aut}}_{k}(R_{n,e,P})\cong\mathcal{A}_{1}=\mathcal{A}_{3}$.
In the case $n\neq1$, $\mathrm{\mathrm{Aut}}_{k}\left(B_{n,P}\right)\cong\mathcal{A}_{2}$.
The isomorphism of $\mathrm{\mathrm{Aut}}_{k}(R_{n,e,P})$ to $\mathcal{A}_{1}=\mathcal{A}_{3}$
is induced by restriction of any automorphism of $R_{n,e,P}$ to the
\textup{\emph{$\mathrm{\mathrm{AL}}$-invariant.}}\end{cor}
\begin{proof}
Theorem \ref{Theo:auto.new example } implies that every algebraic
$k$-automorphism of $R_{n,e,P}$ restricts to an algebraic $k$-automorphism
of $\mathrm{\mathrm{AL}}(R_{n,e,P})=k[X,S]$ that preserves the ideals
$\langle X\rangle$ and $I_{1}$ (resp. every algebraic $k$-automorphism
of $B_{n,P}$ restricts to an algebraic $k$-automorphism of $\mathrm{\mathrm{AL}}(R_{n,e,P})=k[X,S]$
that preserves the ideals $\langle X\rangle$ and $I_{2}$). Finally,
since $R_{n,e,P}$ is the affine modification of the $\mathrm{\mathrm{AL}}$-invariant
along $X^{nm+e}$ with center $I_{1}$ (resp. $B_{n,P}$ is the affine
modification of the $\mathrm{\mathrm{AL}}$-invariant along $X^{n}$
with center $I_{2}$), see Proposition \ref{prop:affine-mod}, every
algebraic $k$-automorphism of $\mathrm{\mathrm{AL}}(R_{n,e,P})$
that preserves the ideals $\langle X\rangle$ and $I_{1}$ (resp.
preserves the ideals $\langle X\rangle$ and $I_{2}$) extends in
a unique way to an algebraic $k$-automorphism of the affine modification
$R_{n,e,P}$ (resp. $B_{n,P}$), see \cite[Corollary 2.2]{Sh. Kaliman and M. Zaidenberg}.
\end{proof}

\subsection{Isomorphism class for the new family \label{sub:Isomorphisms-class}}

\indent\newline\noindent  Again, we only deal with the case where
$Q(X,Y)=Y^{m}$. First, in Proposition \ref{prop:iso non-tivial case}
we deal with the non-trivial case of $R_{n,e,P}$ where ($e\neq0$).
We deliberately exclude the trivial case $e\neq0$, which correspond
to Danielewski $k$-domains of the form $B_{n,P}$ ($\simeq R_{n,0,P,m}$).
Then, in Proposition \ref{prop:not iso trivial case} we compare the
non-trivial case of $R_{n,e,P}$ ($e\neq0$) with the trivial case
$B_{n,P}$ ($\simeq R_{n,0,P,m}$). The reason for doing that is to
elaborate the importance of the non-trivial case $R_{n,e,P}$ where
($e\neq0$), that is, they are not isomorphic to any of Danielewski
$k$-domains.

\subsubsection{\emph{The case} \emph{$R_{n,e,P}$ where $($$e\neq0$$)$}}

Let $R_{n,e,P}$ be the ring defined as 
\[
R_{n,e,P}:=k[X,Y,Z]/\left\langle X^{n}Y-P(X,Y^{m}-X^{e}Z)\right\rangle 
\]
 where $n,e\geq1$, $P(X,S)=S^{d}+f_{d-2}(X)S^{d-2}+\cdots+f_{1}(X)S+f_{0}(X)$,
and $d,m>1$. 

We give necessary and sufficient conditions for two rings, of the
form $R_{n,e,P}$, to be isomorphic. Let $P_{1}(X,S)=S^{d_{1}}+f_{d-2}(X)S^{d_{1}-2}+\cdots+f_{1}(X)S+f_{0}(X)$
and $P_{2}(X,S)=S^{d_{2}}+g_{d-2}(X)S^{d_{2}-2}+\cdots+g_{1}(X)S+g_{0}(X)$.
Then we have the following 
\begin{prop}
\label{prop:iso non-tivial case} $R_{n_{1},e_{1},P_{1},m}$ $\simeq$
$R_{n_{2},e_{2},P_{2},m}$ if an only if $n=n_{1}=n_{2}$, $e=e_{1}=e_{2}$,
$d=d_{1}=d_{2}$, and there exist $\lambda,\mu\in k^{*}$ such that
$f_{d-i}(\lambda X)\equiv\mu^{i}g_{d-i}(X)\mod\, X^{n+e}$, and $\frac{\mu^{dm}}{\lambda^{nm}}=\mu$
for every $i\in\{2,\ldots,d\}$.\end{prop}
\begin{proof}
Let $x_{i},s_{i},y_{i},z_{i}$ be the class of $X,\, S=Y^{m}-X^{e_{i}}Z,\,\, Y$,
and $Z$ in $R_{n_{i},e_{i},P_{i},m}$ for $i\in\{1,2\}$. Let $\Psi:R_{n_{1},e_{1},P_{1},m}\longrightarrow R_{n_{2},e_{2},P_{2},m}$
be an isomorphism between the two semi-rigid rings. Then it induces
$\psi$ an automorphism of $R_{n_{2},e_{2},P_{2},m}$, which restricts
by Corollary \ref{cor:aut.restricts} to an automorphism of 
\[
k[x_{2},s_{2},y_{2}]=B_{n_{2},P_{2}}:=k[X,S,Y]/\left\langle X^{n_{2}}Y-P_{2}(X,S)\right\rangle \subset R_{n_{2},e_{2},P_{2},m}.
\]
Also, $\psi$ restricts to an automorphism of $k[x_{2},s_{2}]$ (resp.
$k[x_{2}]$).

Proposition \ref{prop:iso-preserve-filtration} shows that $\Psi$
respects the semi-rigid structure, that is, $\Psi\left(\mathcal{F}_{j}\right)=\mathcal{G}_{j}$
for every $j$, where \textit{$\{\mathcal{F}_{j}\}_{j\in\mathbb{N}}$}
(resp. \textit{$\{\mathcal{G}_{j}\}_{j\in\mathbb{N}}$} ) is the unique
$\mathrm{LND}$-filtration of $R_{n_{1},e_{1},P_{1},m}$ (resp. $R_{n_{2},e_{2},P_{2},m}$
). Therefore, $\Psi$ restricts to an isomorphism between $k[x_{1},s_{1},y_{1}]=B_{n_{1},P_{1}}$
and $k[x_{2},s_{2},y_{2}]=B_{n_{2},P_{2}}$ Also, $\Psi$ restricts
to an isomorphism between $k[x_{1},s_{1}]$ and $k[x_{2},s_{2}]$
(resp. $k[x_{1}]$ and $k[x_{2}]$). So we conclude that $\Psi(x_{1})=\psi(x_{2})$,
$\Psi(s_{1})=\psi(s_{2})$, $\Psi(y_{1})=\psi(y_{2})$, and $\Psi(z_{1})=\psi(z_{2})$.
This directly implies that $n=n_{1}=n_{2}$, and $d=d_{1}=d_{2}$.

In addition, Theorem \ref{Theo:auto.new example } fully describes
$\psi$, so we get the following form of $\Psi$:
\[
\begin{cases}
\Psi(x_{1})=\psi(x_{2})=\lambda x_{2}\\
\Psi(s_{1})=\psi(s_{2})=\mu s_{2}+x_{2}^{n+e_{2}}a(x_{2})\\
\Psi(y_{1})=\psi(y_{2})=\frac{\mu^{d}}{\lambda^{n}}y_{2}+W\\
\Psi(z_{1})=\psi(z_{2})=\frac{\mu^{dm}}{\lambda^{nm+e_{2}}}z_{2}+\frac{(\frac{\mu^{d}}{\lambda^{n}}y_{2}+W)^{m}-\frac{\mu^{dm}}{\lambda^{nm}}y_{2}^{m}+x_{2}^{n+e_{2}}a(x_{2})}{\lambda^{e_{2}}x_{2}^{e_{2}}}
\end{cases}
\]
for certain $\mu,\lambda\in k*$ such that $\frac{\mu^{dm}}{\lambda^{nm}}=\mu$
and $g_{d-i}(\lambda x_{2})\equiv\mu^{i}g_{d-i}(x_{2})\mod\, x_{2}^{n+e_{2}}$
for every $i\in\{2,\ldots,d_{2}\}$, $a(x_{2})\in k[x_{2}]$, and
$W:=\frac{P_{2}(\lambda x_{2},\mu s_{2}+x_{2}^{n+e_{2}}a(x_{2}))-\mu^{d}P_{2}(x_{2},s_{2})}{\lambda^{n}x_{2}^{n}}.$

Finally, applying $\Psi$ to the relation $x_{1}^{e_{1}}z_{1}=y_{1}^{m}-s_{1}$
in $R_{n,e_{1},P_{1},m}$, we get $\lambda^{e_{1}}x_{2}^{e_{1}}\Psi(z_{1})=\frac{\mu^{md}}{\lambda^{mn}}[y_{2}^{m}-s_{2}]+b=\frac{\mu^{md}}{\lambda^{mn}}[x_{2}^{e_{2}}z_{2}]+b$
where $b\in\mathcal{G}_{md-1}$. 

Comparing top homogeneous components, relative to the filtration \textit{$\{\mathcal{G}_{j}\}_{j\in\mathbb{N}}$,}
for the last equation, we obtain $e=e_{1}=e_{2}$. And we are done.
\end{proof}

\subsubsection{\emph{Comparison with Danielewski $k$-domains}}

\indent\newline\noindent  The next Proposition shows that rings of
the new family $R_{n,e,P}$ (\emph{$e\neq0$}) are not isomorphic
to any of Danielewski $k$-domains. 
\begin{prop}
\label{prop:not iso trivial case} The ring $R_{n_{1},e,P,Q}=k[X,Y,Z]/\left\langle X^{n_{1}}Y-P\left(X,Q(X,Y)-X^{e}Z\right)\right\rangle $
is not isomorphic to $B_{n_{2},F}=k[X,S,Y]/\langle X^{n_{2}}Y-F(X,S)\rangle$
for any $n_{1},n_{2},e>0$, and $P(X,S),Q(X,S),F(X,S)\in k[X,S]$
such that $\deg_{S}F,\deg_{S}Q,\deg_{S}P\geq2$.\end{prop}
\begin{proof}
The case where $n_{2}=1$ is obvious since $\mathrm{ML}(B_{1,P})=k$
which yields by Proposition \ref{prop:semi-rigid=00003DML=00003Dker}
that $B_{1,P}$ is not semi-rigid, while $R_{n_{1},F,Q}$ is for any
$n_{1}\geq1$, see Corollary \ref{Coro:semi-rigidity- of-the-general-case}.

Suppose that $n_{2}\geq2$, then both $B_{n_{2},P}$ and $R_{n_{1},e,F,Q}$
are semi-rigid $k$-domains. Let $x_{1},s_{1},y_{1},z$ be the class
of $X$, $S=Q(X,Y)-X^{e}Z$, $Y$, and $Z$ in $R_{n_{1},F,Q}$, and
let $x_{2},s_{2},y_{2}$ be the class of $X,\, S$, and $Y$ in $B_{n_{2},P}$.
denote by \textit{$\{\mathcal{F}_{i}\}_{i\in\mathbb{N}}$} (resp.
\textit{$\{\mathcal{G}_{i}\}_{i\in\mathbb{N}}$} ) the unique proper
$\mathbb{N}$-filtration of $R_{n_{1},e,F,Q}$ (resp. $B_{n_{2},P}$).

Let $\Psi:B_{n_{2},P}\longrightarrow R_{n_{1},e,F,Q}$ be an isomorphism
between the two rings, then $\Psi$ must respect the semi-rigid structure
of the two rings, that is, $\Psi\left(\mathcal{G}_{i}\right)=\mathcal{F}_{i}$
for every $i$, see \ref{sub:Isomorphisms-class}. This immediately
implies that $\Psi$ restricts to an isomorphism between $\mathrm{\mathrm{AL}}$-invariants
$k[x_{2},s_{2}]\underset{\Psi}{\simeq}k[x_{1},s_{1}]$. 

On the other hand, we have $y_{2}\in\mathcal{G}_{l}$, $y_{1}\in\mathcal{F}_{d}$,
$z\in\mathcal{F}_{md}$, where $\deg_{S}F=l,\deg_{S}Q=m,\deg_{S}P=d$.
Assume that $d\lneq l$, then there exists an element $b\in k[x_{2},s_{2}]$
such that $\Psi(b)=y_{1}$, which means that $y_{1}\in k[x_{1},s_{1}]$,
a contradiction. In the same way we get a contradiction if we assumed
that $l\lneq d$. So the only possibility is $d=l$, thus we conclude
that $k[x_{2},s_{2},y_{2}]\underset{\Psi}{\simeq}k[x_{1},s_{1},y_{1}]$.
Finally, let $b\in B_{n_{2},P}$ such that $\Psi(b)=z$. Since $\Psi(b)\in k[x_{1},y_{1},s_{1}]$,
we get $z\in k[x_{1},y_{1},s_{1}]$, which is a contradiction ($e\geq1$).
And we are done. 
\end{proof}

\section{\textbf{Cylinders over the new class }}

In this section we are interested in finding an algorithm to construct
an explicit isomorphism between cylinders over certain member of the
new family rather than stating that such cylinder are isomorphic.
The latter is known to be true in the abstract due to the classic
Danielewski argument.

We will create explicit isomorphisms between cylinders over rings
of the form $R_{n,e}$ defined by: 
\[
R_{n,e}:=R_{n,e,S^{2}+1,Y^{2}}=k[X,Y,Z]/\langle X^{n}Y-(Y^{2}-X^{e}Z)^{2}-1\rangle
\]
where $e\geq0$, $n\geq1$, and $(n,e)\neq(1,0)$\emph{.}

\subsection{Basic strategy \label{sub:Basic-properties}}

\indent\newline\noindent  Let $\Phi:k^{[N]}\longrightarrow k^{[N]}$
be an endomorphism of $k^{[N]}=k[X_{1},\ldots,X_{N}]$ and let $F\in k^{[N]}$
be an irreducible polynomial. Let $G$ be an irreducible factor of
$\Phi(F)$ in $k^{[N]}$, so we have $\Phi(F)\in\langle G\rangle$.
Consider the induced homomorphism of algebras $\Phi^{*}:k^{[N]}\longrightarrow k^{[N]}/\langle G\rangle$
given by $\Phi^{*}=\pi_{G}\circ\Phi$ where $\pi_{G}:k^{[N]}\longrightarrow k^{[N]}/\langle G\rangle$
is the natural projection. Notice that $\mathrm{Im}(\Phi^{*})\simeq k^{[N]}/\ker(\Phi^{*})$.

Now, suppose that $\Phi^{*}$ is surjective, then $\ker(\Phi^{*})=\langle F\rangle$.
Indeed, if $\Phi^{*}$ is surjective then $\mathrm{Im}(\Phi^{*})\simeq k^{[N]}/\langle G\rangle\simeq k^{[N]}/\ker(\Phi^{*})$
which implies in particular that the ideal $\ker(\Phi^{*})\subset k^{[N]}$
is principal. But since $\Phi^{*}(F)=0$, $\langle F\rangle\subset\ker\Phi^{*}$,
and $F$ is irreducible, we conclude that $\langle F\rangle=\ker\Phi^{*}$.

The latter shows that an isomorphism between $k^{[N]}/\langle F\rangle$
and $k^{[N]}/\langle G\rangle$ can be obtained if we find an endomorphism
of $k^{[N]}$ that verify: first $\Phi(F)\in\langle G\rangle$ (or
simply $\Phi(F)=G$), and second $\Phi^{*}=\pi_{G}\circ\Phi$ is surjective.

\subsection{The case $e\neq0$}

\indent\newline\noindent  First we will start with a simple case
and then we proceed to the general case by induction. We should mention
that following results are known abstractly, however here we give
an algorithm.
\begin{lem}
\label{Lem:R_1,1-R_1,2} $R_{1,1}\otimes_{k}k[T]\simeq R_{1,2}\otimes_{k}k[T]$.\end{lem}
\begin{proof}
$($Which is also an \textbf{\large algorithm} to construct isomorphisms$)$

Let $\Phi:k^{[5]}\longrightarrow k^{[5]}$ be the endomorphism of
$k^{[5]}=k[X,S,Y,Z,T]$ defined as follows: 
\[
\begin{cases}
\Phi(X)=X\\
\Phi(S)=S+H(X,T)\\
\Phi(Y)=Y+L(X,S,T)\\
\Phi(Z)=XZ+F(X,S,Y,T)\\
\Phi(T)=T(X,S,Y,Z,T)
\end{cases}
\]
We choose $H=H(X,T)$ such that:

a) $\Phi(XY-S^{2}-1)=XY-S^{2}-1$. This gives the following relation
between $H$ and $L=L(X,S,T)$:
\[
XL=2HS+H^{2}
\]
which implies that $X$ divides $H$. So $H=XH_{1}$, hence $L=2H_{1}S+XH_{1}^{2}$.

b) $\Phi(Y^{2}-XZ-S)=Y^{2}-X^{2}Z-S$. This gives the following relation
between $H$, and $F=F(X,S,Y,T)$: 
\[
XF=2Y(2H_{1}S+XH_{1}^{2})+(2H_{1}S+XH_{1}^{2})^{2}-XH_{1}
\]
which directly implies that $X$ divides $H_{1}$, and we obtain $H=X^{2}H_{2}$
for some $H_{2}\in k[X,T]$.

Note that if $X^{3}$divides $H$, then we immediately notice that
$\Phi(Z)$ is divisible by $X$ which implies that $\Phi$ does not
induce an isomorphism. Therefore, The only choice for $H$ is such
that $X^{2}$ divides $H$ but $X^{3}$ does not. 

Let for instance $H(X,T)=X^{2}T$ (any other choice such that $H_{2}\in k[T]$
will do), then it reminds to determine $\Phi(T)$ to fully describe
$\Phi$. 

Choose $\Phi(T)$ to be such that the following holds 
\[
\Phi(YZ-XT)=4T(-XY)+4TS(Y^{2}-X^{2}Z)
\]
which simply means that 
\[
\Phi(YZ-XT)\equiv-4T\mod\langle XY-S^{2}-1,Y^{2}-X^{2}Z-S\rangle.
\]
 Note that such a choice of $\Phi(T)$ is always possible even in
a more complicated situation where $P(X,S)$ can be any polynomial
in $k[X,S]$.

An elementary computation can determine $\Phi(T)$ to reach to the
following form of $\Phi$:

(1) $\Phi(X)=X$ 

(2) $\Phi(S)=S+X^{2}T$ 

(3) $\Phi(Y)=Y+2XST+X^{3}T^{2}$

(4) $\Phi(Z)=XZ+4SYT-XT+2X^{2}YT^{2}+4XS^{2}T^{2}+4X^{3}ST^{3}+X^{5}T^{4}$

(5) $\Phi(T)=ZY+6XSZT+3YT+2XY^{2}T^{2}+12S^{2}YT^{2}+X^{3}ZT^{2}-X^{3}T^{3}+12X^{2}SYT^{3}+8XS^{3}T^{3}+3X^{4}YT^{4}+12X^{3}S^{2}T^{4}+6X^{5}ST^{5}+X^{7}T^{6}$

Now, define $\phi:k^{[4]}\longrightarrow k^{[4]}$ to be the endomorphism
of $k^{[4]}=k[X,Y,Z,T]$ given by: $\phi(X)=\Phi(X)$, $\phi(Y)=\Phi(Y)$,
$\phi(Z)=\Phi(Z)$, $\phi(T)=\Phi(T)$ where we substitute $S$ by
$S=Y^{2}-X^{2}Z$. Then we have $\phi(Y^{2}-XZ)=S+X^{2}T$, $\phi\left(XY-(Y^{2}-XZ)^{2}-1\right)=XY-(Y^{2}-X^{2}Z)^{2}-1$,
and $\phi(YZ-XT)=-4T\mod\langle XY-(Y^{2}-X^{2}Z)^{2}-1\rangle$.

As discussed before \S \ref{sub:Basic-properties}, to prove that
$\phi$ induces an isomorphism between $R_{1,1}\otimes_{k}k[T]$ and
$R_{1,2}\otimes_{k}k[T]$, it is enough to show that $\phi^{*}=\pi_{XY-(Y^{2}-X^{2}Z)^{2}-1}\circ\phi$
is surjective. For that, denote by $x,s,y,z,t$ the class of $X,S,Y,Z$,
and $T$ in $R_{1,2}\otimes_{k}k[T]$, then (1) immediately shows
that $x\in\mathrm{Im}\phi^{*}$. By construction $t\in\mathrm{Im}\phi^{*}$.
Therefore, (2) implies $s\in\mathrm{Im}\phi^{*}$, and (3) implies
that $y\in\mathrm{Im}\phi^{*}$. So (4) provides $xz\in\mathrm{Im}\phi^{*}$,
and (5) ensures that $yz\in\mathrm{Im}\phi^{*}$. Since $s=y^{2}-x^{2}z$,
we get $sz=y(yz)-(xz)^{2}$. This means that $sz\in\mathrm{Im}\phi^{*}$.
Finally, since $z=z(xy-s^{2})=x(yz)-s(sz)$ where all terms in the
second part of the last equation belong to $\mathrm{Im}\phi^{*}$,
we deduce that $z\in\mathrm{Im}\phi^{*}$. In conclusion, $\phi^{*}$
is surjective. 
\end{proof}
The exact same algorithm, as in the proof of Lemma \ref{Lem:R_1,1-R_1,2},
can be applied to construct an isomorphism between $R_{1,e}\otimes_{k}k[T]$
and $R_{1,e+1}\otimes_{k}k[T]$ for every $e>0$. The only different
step is that $\Phi(T)$ must be chosen to verify:
\[
\Phi(YZ-XT)=4T(-XY)+4TS(Y^{2}-X^{e+1}Z).
\]
 Also, the same algorithm can be used to establish an isomorphism
between $R_{n,1}\otimes_{k}k[T]$ and $R_{n,2}\otimes_{k}k[T]$ for
every $n>0$, where $\Phi(T)$ is chosen to hold: 
\[
\Phi(YZ-XT)=4T(-X^{n}Y)+4TS(Y^{2}-X^{2}Z).
\]
 We put together the previous observation to obtain the following. 
\begin{lem}
$R_{1,e}\otimes_{k}k[T]$ $\simeq$ $R_{1,e+1}\otimes_{k}k[T]$, and
$R_{n,1}\otimes_{k}k[T]$ $\simeq$ $R_{n,2}\otimes_{k}k[T]$.
\end{lem}
Finally, by induction we get: 
\begin{thm}
\label{Thm:Cylinder over R_n,e_1 iso to Cylinder over R_n,e_2 } $R_{n,e_{1}}\otimes_{k}k[T]\simeq R_{n,e_{2}}\otimes_{k}k[T]$
for every $n,e_{1},e_{2}>0$.

In addition, if $\phi_{e_{1}+i,e_{1}+i+1}$ is the endomorphisms,
as determined in Lemma \ref{Lem:R_1,1-R_1,2}, of $k^{[4]}$ that
induces an isomorphism between $R_{n,e_{1}+i}[T]$ and $R_{n,e_{1}+i+1}[T]$,
then the endomorphisms $\phi_{e_{2}-1,e_{2}}\circ\cdots\circ\phi_{e_{1},e_{1}+1}$
of $k^{[4]}$ induces an isomorphism between $R_{n,e_{1}}[T]$ and
$R_{n,e_{2}}[T]$.
\end{thm}

\subsubsection{\emph{A counter-example of the cancellation problem}}

\indent\newline\noindent  Consider the following chains of inclusions,
which is realized by sending $Z$ to $XZ$ in every step. 
\[
R_{n_{0},1}\hookrightarrow R_{n_{0},2}\hookrightarrow\cdots\hookrightarrow R_{n_{0},e}
\]
for every $n_{0}\in\{1,\ldots,n\}$.

They are pairwise not isomorphic to each other by Proposition \ref{prop:iso non-tivial case},
whereas, Theorem \ref{Thm:Cylinder over R_n,e_1 iso to Cylinder over R_n,e_2 }
indicates 
\[
R_{n_{0},1}\otimes_{k}k[T]\simeq R_{n_{0},2}\otimes_{k}k[T]\simeq\cdots\simeq R_{n_{0},e}\otimes_{k}k[T]
\]
for every $n_{0}\in\{1,\ldots,n\}$.

\subsection{Cylinders over Danielewski $k$-domains, the case $e=0$ }

\indent\newline\noindent  Here, we will show how to create an isomorphism
between cylinders over $k$-domains of the form: 
\[
B_{n,P}=k[X,S,Y]/\langle X^{n}Y-P(X,S)\rangle
\]
for every $n\geq1$. Where $P(X,S)=S^{d}+XQ(X,S)+c$, $c\in k-\{0\}$,
and $Q(X,S)\in k[X,S]$, 

First, we illustrate how the algorithm, presented in the proof of
Lemma \ref{Lem:R_1,1-R_1,2}, can be modified to establish isomorphisms
between the below mentioned rings and then we proceed to the general
case.
\begin{lem}
\label{lem:B_1,P-B_2,P} $B_{1,P}\otimes_{k}k[T]$ $\simeq$ $B_{2,P}\otimes_{k}k[T]$,
where $P(X,S)=S^{4}+X^{2}S^{2}+1$.\end{lem}
\begin{proof}
In a similar way as in the proof of Lemma \ref{Lem:R_1,1-R_1,2},
we will establish an isomorphism between $B_{1,P}[T]$ and $B_{2,P}[T]$. 

Let $\Phi:k^{[4]}\longrightarrow k^{[4]}$ be the endomorphism of
$k^{[4]}$ defined as follows:
\[
\begin{cases}
\Phi(X)=X\\
\Phi(S)=S+H(X,T)\\
\Phi(Y)=XY+L(X,S,T)\\
\Phi(T)=T(X,S,Y,T)
\end{cases}
\]
We choose $H=H(X,T)$ such that $\Phi(XY-S^{4}-X^{2}S^{2}-Xf(X)-1)=X^{2}Y-S^{4}-X^{2}S^{2}-Xf(X)-1$.
This gives the following relation between $H$ and $L=L(X,S,T)$ :
\[
XL=H^{4}+4H^{3}S+6H^{2}S^{2}+4HS^{3}+H^{2}X^{2}+2HSX^{2}
\]
which directly implies that $H$ is divisible by $X$. Note that if
$X^{2}$divides $H$, then we immediately obtain $\Phi(Y)$ is divisible
by $X$ which implies that $\Phi$ will never induces an isomorphism.
Therefore, The only choice for $H$ is such that $X$ divides $H$
but $X^{2}$ does not. So let $H(X,T)=XT$, then a simple computation
leads to 
\[
L(X,S,T)=X^{3}T^{4}+4X^{2}T^{3}S+6XT^{2}S^{2}+4TS^{3}+T^{2}X^{3}+2TSX^{2}.
\]

Choose $\Phi(T)$ such that 
\[
\Phi(YS-XT)=4T(-X^{2}Y+S^{4}+X^{2}S^{2}+Xf(X)),
\]
which can be done by virtue of the condition ($X$ divides every coefficient
of $P(X,S)-S^{4}-1$). Notice that this choice is made to get 
\[
\Phi(YS-XT)\equiv-4T\mod(X^{2}Y-S^{4}-X^{2}S^{2}-Xf(X)-1),
\]
which simply implies that $T+\langle X^{2}Y-P\rangle\in\mathrm{Im}\pi_{X^{2}Y-P}\circ\Phi$.
An elementary computation can determine $\Phi(T)$ to reach to the
following form of $\Phi$:

(1) $\Phi(X)=X$

(2) $\Phi(S)=S+XT$

(3) $\Phi(Y)=XY+X^{3}T^{4}+4X^{2}T^{3}S+6XT^{2}S^{2}+4TS^{3}+T^{2}X^{3}+2TSX^{2}$

(4) $\Phi(T)=SY+5XYT-4f(X)T-2S^{2}TX+3ST^{2}X^{2}+10S^{3}T^{2}+T^{3}X^{3}+10S^{2}T^{3}X+5ST^{4}X^{2}+T^{5}X^{3}$

Now, as we discussed before \S \ref{sub:Basic-properties}, to prove
that $\Phi$ induces an isomorphism between $B_{1,P}$ and $B_{2,P}$,
it is enough to show that $\Phi^{*}=\pi_{X^{2}Y-P}\circ\Phi$ is surjective.
For that, denote by $x,s,y,t$ the class of $X,S,Y$, and $T$ in
$R_{2,0,P}$, then immediately we see that $x\in\mathrm{Im}\Phi^{*}$.
By construction $t\in\mathrm{Im}\Phi^{*}$, therefore (2) implies
$s\in\mathrm{Im}\Phi^{*}$. Thus (3) provides $xy\in\mathrm{Im}\Phi^{*}$,
again using (4) we see that $sy\in\mathrm{Im}\Phi^{*}$. Finally,
since $y=y.1=y(x^{2}y-s^{4}-s^{2}x^{2}-xf(x))=(xy)^{2}-(ys)s^{3}-(ys)sx^{2}-(xy)f(x)$
where all terms in the second part of the last equation belong to
$\mathrm{Im}\Phi^{*}$, we deduce that $y\in\mathrm{Im}\Phi^{*}$.
Since $x,s,y,t$ are generators of $B_{2,P}$, we conclude that $\Phi^{*}$
is surjective.
\end{proof}
The exact same method, as in the proof of Lemma \ref{lem:B_1,P-B_2,P},
can be used to create an isomorphism between $k$-domains presented
in the following Proposition. Where we imposed conditions, $X$ divides
every coefficient of $P(X,S)-S^{d}-c$ as a polynomial in $S$, and
$c\neq0$, which are essential to enable us to proceed.
\begin{lem}
\label{Lem:cylinder iso} $B_{n,P}\otimes_{k}k[T]\simeq B_{n+1,P}\otimes_{k}k[T]$.

where $n\geq1$, $d\geq2$, $c\in k\setminus\{0\}$, $P(X,S)=S^{d}+XQ(X,S)+c$,
and $Q(X,S)\in k[X,S]$ with no restriction on the degree of $Q(0,S)$.
\end{lem}
Finally, by induction we have the following.
\begin{thm}
\label{Cor:iso.D.w.Cylindre} $B_{n,P}\otimes_{k}k[T]\simeq B_{m,P}\otimes_{k}k[T]$.

where $n,m\geq1$, $d\geq2$, $c\in k\setminus\{0\}$, $P(X,S)=S^{d}+XQ(X,S)+c$,
and $Q(X,S)\in k[X,S]$. 

In addition, if $\Phi_{i,i+1}$ is the endomorphisms, as determined
in Lemma \ref{lem:B_1,P-B_2,P}, of $k^{[4]}$ that induces an isomorphism
between $B_{i,P}[T]$ and $B_{i+1,P}[T]$, then the endomorphisms
$\Phi_{m-1,m}\circ\cdots\circ\Phi_{n,n+1}$ of $k^{[4]}$ induces
an isomorphism between $B_{n,P}[T]$ and $B_{m,P}[T]$.
\end{thm}
\indent\newline


\noindent  
\begin{thebibliography}{K-ML1}
{\normalsize \bibitem[C-ML]{key-Crachiola Makar} A. Crachiola and
L. Makar-Limanov, }\emph{\normalsize An algebraic proof of a cancellation
theorem for surfaces}{\normalsize , J. Algebra 320 (2008), no. 8,
3113\textendash{}3119.}{\normalsize \par}

{\normalsize \bibitem[D1]{Daigle: traingulability} D. Daigle, }\emph{\normalsize A
necessary and sufficient condition for triangulability of derivation
of $k[X,Y,Z]$}{\normalsize , J. Pure Appl. Algebra 113 (1996), 297-305. }{\normalsize \par}

{\normalsize \bibitem[D2]{D. daigle:LND-degree} D. Daigle, }\emph{\normalsize Polynomials
$f(X,Y,Z)$ of low $\mathrm{LND}$-degree}{\normalsize , CRM Proceedings
\& Lecture Notes 54 (2011), 21-34.}{\normalsize \par}

{\normalsize \bibitem[D3]{D. daigle: tame degree} D. Daigle, }\emph{\normalsize Tame
and wild degree functions}{\normalsize , Osaka J. of Math. 49 (2012),
53-80.}{\normalsize \par}

{\normalsize \bibitem[F-M]{key-David Maubach} D. Finston and S. Maubach,
}\emph{\normalsize Constructing (almost) rigid rings and a UFD having
infinitely generated Derksen and Makar-Limanov invariant}{\normalsize ,
Canad. Math. Bull.}{\normalsize \par}

{\normalsize \bibitem[F]{key-gene} G. Freudenburg, }\emph{\normalsize Algebraic
Theory of Locally Nilpotent Derivation}{\normalsize , Encycl. Math.
Sci., 136, Inv. Theory and Alg. Tr. Groups, VII, Springer- Verlag,
2006. }{\normalsize \par}

{\normalsize \bibitem[K-ML1]{key-Kaliman Makar} Sh. Kaliman and L.
Makar-Limanov, }\emph{\normalsize AK-invariant of affine domains}{\normalsize ,
Affine Algebraic Geometry}\emph{\normalsize , }{\normalsize pages
231\textendash{}255, Osaka University Press, Osaka, 2007.}{\normalsize \par}

{\normalsize \bibitem[K-ML2]{key-sh and makar} Sh. Kaliman and L.
Makar-Limanov, }\emph{\normalsize On the Russell-Koras contractible
threefolds}{\normalsize , J. Algebraic Geom., 6 no. 2 (1997), 247\textendash{}268.}{\normalsize \par}

{\normalsize \bibitem[K-Z]{Sh. Kaliman and M. Zaidenberg} Sh. Kaliman
and M. Zaidenberg, }\emph{\normalsize Affine modification and affine
hypersurfaces with a very transitive automorphism group}{\normalsize ,
Transform. Groups 4 (1999). no. 1. 53-95.}{\normalsize \par}

{\normalsize \bibitem[ML1]{key-Makar} L. Makar-Limanov, }\emph{\normalsize Locally
nilpotent derivations, a new ring invariant and applications}{\normalsize ,
available at http://www.math.wayne.edu/\textasciitilde{}lml/lmlnotes.pdf.}{\normalsize \par}

{\normalsize \bibitem[ML2]{key-ML2001} L. Makar-Limanov, }\emph{\normalsize On
the group of automorphisms of a surface$x^{n}y=P(z)$}{\normalsize ,
Israel J. Math., 121:113\textendash{}123, 2001.}{\normalsize \par}

{\normalsize \bibitem[ML3]{key-M-L1996} L. Makar-Limanov,}\emph{\normalsize{}
On the hypersuface $x+x^{2}y+z^{2}+t^{3}=0$ in $\mathbb{C}^{4}$
or a $\mathbb{C}^{3}$-like threefold which is not $\mathbb{C}^{3}$}{\normalsize ,
Israel J. Math., 96 (part B), pp. 419-429, 1996.}{\normalsize \par}

{\normalsize \bibitem[P]{Ploni} P.-M. Poloni, }\emph{\normalsize Classification(s)
of Danielewski hypersurfaces}{\normalsize . Transform. Groups 16,
Issue 2 (2011), 579-597.}\end{thebibliography}
\end{document}